\documentclass[10pt, reqno, final]{amsart}

\title{A functional Law of the Iterated Logarithm for weakly hypoelliptic diffusions at time zero} 
\author[M. Carfagnini, J. F\"{o}ldes, D. P. Herzog]
     {Marco Carfagnini, Juraj F\"{o}ldes, David P. Herzog\\
  \scriptsize{emails: marco.carfagnini@uconn.edu, foldes@virginia.edu, 
                dherzog@iastate.edu}
}

\thanks{Our efforts were partially supported under grants
 NSF-DMS-1712427 (MC),  NSF-DMS-1954264 (MC),
NSF-DMS-1816408 (JF), 
 NSF-DMS-1855504 (DPH)} 

\usepackage{amsfonts, amssymb, amsmath, amsthm,enumerate,esint,multicol, mathptmx}
\usepackage[colorlinks=true, pdfstartview=FitV, linkcolor=blue,
            citecolor=blue, urlcolor=blue]{hyperref}
\usepackage[usenames]{color}
\definecolor{Red}{rgb}{0.7,0,0.1}
\definecolor{Green}{rgb}{0,0.7,0}
\usepackage{accents}
\usepackage{comment}
\usepackage{graphicx,wrapfig,tikz}
\usetikzlibrary{calc, hobby}
\usepackage{pgf}  
\usetikzlibrary{decorations.pathmorphing}
\usepackage[capitalize,nameinlink,noabbrev]{cleveref}
\usepackage[notcite, notref]{showkeys}
\usepackage[a4paper, total={6in, 8in}]{geometry}
\numberwithin{equation}{section}
\newtheorem{Theorem}{Theorem}[section]
\newtheorem{Proposition}[Theorem]{Proposition}
\newtheorem{Lemma}[Theorem]{Lemma}
\newtheorem{Corollary}[Theorem]{Corollary}

\theoremstyle{definition}
\newtheorem{Definition}{Definition}[section]
\newtheorem{Example}{Example}[section]
\newtheorem{Assumption}{Assumption}

\makeatother\newtheorem{Remark}[Theorem]{Remark}

\newcommand{\PP}{\mathbf{P}}

\newcommand{\N}{\mathbf{N}}

\newcommand{\Y}{\mathcal{Y}}
\newcommand{\s}{\text{stoc}}
\newcommand{\K}{\mathcal{K}}

\newcommand{\E}{\mathbf{E}}

\newcommand{\RR}{\mathbf{R}}
\newcommand{\C}{\mathcal{C}}

\newcommand{\EE}{\mathcal{E}}

\newcommand{\A}{\mathcal{A}}
\newcommand{\inter}{\text{interior}}
\newcommand{\ustar}{U^*}

\begin{document}  \begin{abstract}
We study the almost sure behavior of solutions of stochastic differential equations (SDEs) as time goes to zero. Our main general result establishes a functional law of the iterated logarithm (LIL)
that applies in the setting of SDEs with degenerate noise satisfying the \emph{weak H\" ormander condition} but not the \emph{strong H\"{o}rmander condition}.  That is, SDEs in which the drift terms must be used in order to conclude hypoellipticity. As a corollary of this result, we obtain the almost sure behavior as time goes to zero of a given direction in the equation, even if noise is not present explicitly in that direction.  The techniques used to prove the main results are based on large deviations applied to a non-trivial rescaling of the
original system. In concrete examples, we show how to find the proper rescaling to obtain the functional LIL.  Furthermore, we apply the main results to the problem of identifying regular points for hypoelliptic diffusions.  Consequently, we obtain a control-theoretic criteria for a given point to be regular for the process.  \end{abstract}

\maketitle

\section{Introduction}
The law of the iterated logarithm (LIL) for an i.i.d. sequence of random variables $X_1, X_2, \ldots, X_n, \ldots$ with mean zero and unit variance reads
\begin{align}
\label{eqn:LILiid}
\limsup_{n\rightarrow \infty} \frac{\pm S_n}{\sqrt{2n \log \log n}} =1 \quad \text{a.s.}\,,
\end{align} 
where $S_n := X_1 + X_2 + \dots +X_n$ and $\log$ denotes the natural logarithm.  The formula~\eqref{eqn:LILiid}  was first established by Hartman and Whintner in 1941~\cite{HW_41} as a generalization of earlier works of Khinchin in 1924~\cite{Khin_24} and Kolmogorov in 1929~\cite{Kol_29}.  Analogously, an LIL holds for a standard, real-valued Brownian motion $W_t$ as $t\rightarrow \infty$ by replacing $n$ by $t$ and $S_n$ by 
$W_t$ in~\eqref{eqn:LILiid}.  Furthermore, one can use~\eqref{eqn:LILiid} for the Brownian motion at time infinity to obtain an LIL at time zero
\begin{align}
\label{eqn:LILbm}
\limsup_{t\rightarrow 0^+} \frac{\pm W_t}{\sqrt{2t \log \log t^{-1}}}= 1 \quad \text{a.s.}
\end{align}
by Brownian inversion.  

Note that LILs provide an asymptotic window, for example $[-\sqrt{2t \log \log t^{-1}}, \sqrt{2t \log \log t^{-1}}]$ for the process $W_t$ as $t\rightarrow 0^+$, complementing the usual central limit scaling.  More precisely, it follows that the set of limit points of the scaled processes 
$S_n/ \sqrt{2n\log \log n}$ as $n\rightarrow \infty$ or 
$W_t/\sqrt{2t \log \log t^{-1}}$ as $t\rightarrow 0^+$ is the interval $[-1,1]$~\cite{HW_41}.  A further generalization of this limit set analysis is due to Strassen~\cite{Stras_64}, which for $W_t$ at $t=0$ establishes that, for almost every $\omega$, the set of limit points (in the space of continuous paths $\C([0,1]; \RR)$)
of the family  
\begin{align}
\label{eqn:scaledbm}
Y_t^\epsilon(\omega):= \frac{W_{\epsilon t}(\omega)}{\sqrt{2\epsilon \log \log \epsilon^{-1}}}, \,\,\,\, t\in [0,1], 
\end{align}
as $\epsilon \rightarrow 0^+$ is the set of functions $f\in \C^0([0,1]; \RR)$ with $f_0=0$ and $\int_0^1 |\dot{f}_s|^2 \, ds \leq 1$.  Here, 
$$
\C^0([0,1]; \RR^k):= \{ f \in \C([0,1]; \RR^k)\,: \, \dot{f}\in L^2([0,1])\}.
$$
Observe that, by the fundamental theorem of calculus and Jensen's inequality, the condition $\int_0^1 |\dot{f}_s|^2 \, ds \leq 1$ implies that $|f_1| \leq 1$.  By choosing $f_s=\pm s$, the extremal values $\pm 1$ of the pointwise limit set $[-1,1]$ are attained, and, by setting $t = 1$ in~\eqref{eqn:scaledbm}, we obtain~\eqref{eqn:LILbm} as a corollary of Strassen's result.    

The goal of this paper is to provide a framework for establishing Strassen-type LILs at time zero that applies in the setting of \emph{weakly hypoellptic} diffusions.  To formulate our results,  
we fix positive integers $d,k\in \N$, a non-empty, open set $U\subset \RR^d$, and consider an It\^{o} stochastic differential equation (SDE) on $U$ of the form\begin{align}
\label{eqn:sdemain}
\begin{cases} dx_t = \tilde{b}(x_t) \, dt + \tilde{\sigma}(x_t) \, dB_t \,,\\
x_0=x \in U , 
\end{cases}
\end{align}
where $B_t$ is a standard, $k$-dimensional Brownian motion defined on a probability space $(\Omega, \mathcal{F}, \PP)$.  In \eqref{eqn:sdemain}, 
we assume $\tilde{b}:U\rightarrow \RR^d$ and $\tilde{\sigma}: U\rightarrow M_{d\times k}$ are \emph{locally Lipschitz} on $U$; that is, Lipschitz continuous on every compact subset of 
$U$, while $M_{d\times k}$ denotes the set of $d\times k$ matrices with real-valued entries.  Under these assumptions, the solution~\eqref{eqn:sdemain} can be defined pathwise until the first time $\tau(x_\cdot)$ at which $x_t$ exits $U$.  If $\tau(x_\cdot)$ is finite with positive probability, we fix a \emph{death state} $\Delta \notin U$ and set $x_t = \Delta$ for $t\geq \tau(x_\cdot)$.  

In this paper, we focus on degenerate diffusions, or equivalently on matrices $\tilde{\sigma} (x)$ that have rank strictly less than $d$ at the initial condition $x$,  
 so that the dynamics $x_t$ defined by~\eqref{eqn:sdemain} is not trivially dominated at time zero in every direction by the process $x+ \tilde{\sigma}(x) B_t$.  Otherwise, a functional LIL can be readily obtained by rescaling the equation~\eqref{eqn:sdemain} according to the LIL Brownian scaling~\eqref{eqn:scaledbm} and passing to a (functional) limit using the theory of large deviations~\cite{Az_80, Bal_88, Bal_86}.  To  see why, 
 suppose for simplicity that $\tilde{\sigma}$ is a constant $d\times k$ matrix on $U$ and all expressions below are well-defined. Then,  the rescaled process
\begin{align}
\label{eqn:scaling1}
y^\epsilon_t := x+\frac{x_{\epsilon t} -x}{\sqrt{2\epsilon \log \log \epsilon^{-1}}}, \,\,\,\, t\in [0,1], \,\epsilon >0,
\end{align}  
satisfies the integral equation
\begin{align}
\label{eqn:intyeps}
y_t^\epsilon = x+ \int_0^t \frac{\sqrt{\epsilon}}{\sqrt{2 \log \log \epsilon^{-1}}} b(x_{\epsilon s}) \, ds +  \frac{ \tilde{\sigma} B_{\epsilon t}}{\sqrt{2 \epsilon \log \log \epsilon^{-1}}}.  
\end{align}
Since the integral in~\eqref{eqn:intyeps} is small in $\epsilon$, one then expects, and indeed it can be proved that, for almost all $\omega$, the set of limit points of $y^\epsilon_\cdot(\omega)$ in $\C([0,1] ; \RR^k)$ as $\epsilon \rightarrow 0$ is 
precisely the set of $g\in \C^0([0,1];\RR^d)$ of the form 
\begin{align}
g_t &= x+ \tilde{\sigma} f_t
\end{align}
for some $f\in \C^0([0,1]; \RR^k)$ with $f_0=0$ and $\int_0^1 |\dot{f}_s|^2 \, ds \leq 1$. Thus if $\text{rank}(\tilde{\sigma}(x)) = d$, then the asymptotic behavior of every component of $y_t^\epsilon$, and consequently every component of $x_t$, 
can be readily characterized. 

On the other hand, if $\text{rank}(\tilde{\sigma}(x))<d$, then the same rescaling~\eqref{eqn:scaling1} is valid but less informative.  In particular, for the directions that are in the range of $\tilde{\sigma}(x)$, the argument above gives 
correct asymptotic behavior.  However, for the directions in $\RR^d$ perpendicular to the range of $\tilde{\sigma}(x)$, the limiting trajectories are constant almost surely, meaning that the dynamics restricted to this subspace is finer and thus a different scaling is required.  Even if one can heuristically estimate the right scale in the perpendicular directions, more work is needed to establish Strassen's law for the corresponding process.

In this direction, more general Strassen-type LILs for stochastic differential equations~\eqref{eqn:sdemain} 
can be found in the pioneering works of Baldi~\cite{Bal_86} (at time infinity) and later Caramellino \cite{Car_98} (at time zero).  Both works are limited by the permissible scaling transformations (called a \emph{sequence of contractions}) applied to the diffusion to obtain the LIL.  In particular, such scalings do not allow for \emph{weakly hypoelliptic} diffusions; that is, diffusions that satisfy the \emph{weak H\"{o}rmander} condition but not the \emph{strong H\"{o}rmander} condition (see Section~\ref{sec:examples} for further information).  Intuitively, weakly hypoelliptic diffusions are those in which the noise must spread through the drift so that the process does not live on a lower-dimensional manifold of $U$ in small times.  On the other hand, strongly hypoelliptic diffusions are those in which noise only spreads through the diffusion matrix, i.e. the drift is not needed, in order to reach all directions of the phase space.  Weakly hypoelliptic diffusions arise in a number of natural settings, from finite-dimensional stochastic models in turbulence, see~\cite{Bec_05, Bec_07, BK_20, FGHH_20, GHW_11, HN_21, Rom_04, EMat_01, ZelHai_21}, to canonical models in statistical mechanics and machine learning, see~\cite{Cam_21, Cun_18, Gia_19, HerMat_19, LSS_20, Stoltz_10}, where local almost sure behavior at time zero is not understood precisely but is nevertheless important.  Usually, one can estimate the behavior at time zero by finding the support of the process using control theory via the Support Theorems~\cite{SV_72ii, SV_72i}.  This usually gives the behavior of the process in small times with positive, usually very small, probability~\cite{HerMat_15ii}.  The LILs deduced here provide a refinement of the support of the diffusion in small times.

In this paper, we improve upon the main results in~\cite{Bal_86, Car_98} by showing more general scaling transformations are permitted to obtain the functional LIL using the theory of large deviations. Specifically, our methods allow for different scalings in each component which need not be functions of the standard elliptic/Brownian scaling $\sqrt{\epsilon \log \log \epsilon^{-1}}$ as in~\cite{Bal_86, Car_98}, but can be general regularly varying functions in $\epsilon$ as $\epsilon \rightarrow 0^+$.  See Definition~\ref{def:1} in conjunction with Definition~\ref{def:2} below.    In addition, in the setting of SDEs with additive noise, our transformations can in fact be time dependent (see Example~\ref{ex:IK2}), which is needed for degenerate diffusions with non-vanishing drift at the initial condition.  There, because of the time dependence, the rescaled problem has a more complicated structure, which is why we restrict to the case of additive noise.  Furthermore, we apply the main results to several nontrivial examples, e.g. the Iterated Kolmogorov equation in general dimensions and a stochastic Lorenz '96 model, capturing the a.s. behavior in each of these equations in each direction as time goes to zero started from the origin.

Our original motivation for LILs in the setting of hypoelliptic diffusions was to derive a probabilistic method for determining whether certain boundary points are \emph{regular} or \emph{irregular} (cf. Section~\ref{sec:regular}).  Such information is crucial when solving second-order linear hypoelliptic boundary-value problems, for example the Dirichlet or Poisson problems, in a domain~\cite{Ok_03}.  Understanding when a particular point is regular or irregular is a long-standing open question and we refer the reader to the work of Kogoj~\cite{Kog_12, Kog_17} which employs analytical methods from PDEs to provide sufficient conditions for classical solvability of the Dirichlet problem and Harnack-type estimates.  In this paper, we will use the LIL to derive a control theoretic condition for a given boundary point to be regular or irregular (see Section~\ref{sec:regular}).

The organization of this paper is as follows.  In Section~\ref{sec:setup}, we outline notation and state our main general results while in Section~\ref{sec:examples} we apply these results to concrete examples.  We recommend the reader not familiar with the methods to  first loosely read Section~\ref{sec:examples} to obtain some ideas of how to arrive at an LIL in the weakly hypoelliptic setting before reading Section~\ref{sec:setup}.  Section~\ref{sec:ld} outlines the needed results from the theory of large deviations to establish the main general results pertaining to the LILs, which are proved in Section~\ref{sec:proof} and Section~\ref{sec:lem}.  In Section~\ref{sec:regular}, we derive our criteria for a point on the boundary to be regular or irregular.  There, we also discuss applications of this criteria to the design of piecewise $C^1$ boundaries on which all points are regular.

\section{Setting, Notation and Main Results}
\label{sec:setup}
The setup in this section is similar to that in Baldi~\cite{Bal_86} and Caramellino~\cite{Car_98}, but with several differences.  First,  
our setting is slightly more general, which allows for more general transformations of the original process; that is, the process that satisfies~\eqref{eqn:sdemain} above.  Second, 
we work primarily in the space of \emph{explosive paths}, as opposed to continuous paths, defined below, and consequently, we employ notation and results from Azencott~\cite{Az_80}.  Importantly, this allows 
us to work around issues of finite-time explosion in both the SDE and its limiting ODE.  In particular, we can remove Assumption (A) (iii) of~\cite{Car_98}, although in many examples this condition is satisfied.   

\subsection{Laws of the iterated logarithm and large deviations}
The crucial ingredient in the proof of the functional LIL is a change of coordinates, i.e. a rescaling, for the system~\eqref{eqn:sdemain}, allowing one to reformulate the problem using the theory of large deviations.  Unless the diffusion~\eqref{eqn:sdemain} is uniformly elliptic at the initial state $x$, the change of coordinates varies depending on the dynamics.  
To see how to construct such a transformation, we provide concrete examples in Section~\ref{sec:examples}.         

Unlike in~\cite{Bal_86, Car_98}, instead of assuming a particular transformation, 
we simply associate to equation~\eqref{eqn:sdemain} a small parameter $\epsilon_*>0$ and family of processes $\{y^\epsilon\}_{\epsilon\in (0, \epsilon_*]}$ satisfying an SDE of the form 
\begin{align}
\label{eqn:sdeas}
\begin{cases}
d y_t^\epsilon = b_\epsilon(y_t^\epsilon) \, dt + \frac{1}{\sqrt{r(\epsilon)}} \sigma_\epsilon(y_t^\epsilon) \, dB_t^\epsilon,\\
 y_0^\epsilon = x,
 \end{cases}
\end{align}
where for each $\epsilon \in (0, \epsilon_*]$, $B_t^\epsilon$ is a standard $k$-dimensional Brownian motion,  
\begin{align}
r(\epsilon):= \log \log \epsilon^{-1}
\end{align}
 and the coefficients $b_\epsilon, \sigma_\epsilon$ satisfy the following conditions.
 
\begin{Assumption}\label{assump:1}~
There exist $\epsilon_* > 0$, a non-empty open set $\ustar \subset \RR^d$ (not necessarily the same as $U$) and locally Lipschitz  functions $b:\ustar\rightarrow \RR^d$ and $\sigma: \ustar\rightarrow M_{d\times k}$ such that the following properties hold. 
\begin{itemize}
\item[(i)]  For every $\epsilon \in (0, \epsilon_*]$, there is an open set $U_\epsilon \subset \RR^d$ such that  $U ^*\subset U_\epsilon$ and the coefficients $b_\epsilon: U_\epsilon \rightarrow \RR^d$, $\sigma_\epsilon: U_\epsilon\rightarrow M_{d\times k}$ 
are locally Lipschitz on $U_\epsilon$. 
\item[(ii)]  For every compact $K \subset \ustar$, 
 \begin{align*}
\lim_{\epsilon \rightarrow 0^+}  \sup_{y \in K} |b_\epsilon(y) - b(y)| = 0 \,, \\ 
\lim_{\epsilon \rightarrow 0^+} \sup_{y \in K} \|\sigma_\epsilon(y) - \sigma(y) \| = 0 \,,
\end{align*}   
where $\|\cdot \|$ denotes a matrix norm.  
\end{itemize}
\end{Assumption}
In Figure~1, we have provided a sketch of all of some of the objects introduced thus far, with a generic ``mapping" $\Phi_\epsilon$ relating the two processes, $x_t$ and $y_t^\epsilon$.  
\begin{Remark}
Note that $b$ and $\sigma$ in Assumption \ref{assump:1} are not the same as $\tilde{b}$ and $\tilde{\sigma}$ in equation~\eqref{eqn:sdemain}.  One should think of~\eqref{eqn:sdeas} 
as a rescaled version of~\eqref{eqn:sdemain}, where $b_\epsilon$ and $\sigma_\epsilon$ depend on $\tilde{b}$ and $\tilde{\sigma}$.   
\end{Remark}

\begin{Remark}
When comparing the noise terms in~\eqref{eqn:intyeps} and in~\eqref{eqn:sdeas}, note that the $\sqrt{\epsilon}$ in the denominator in~\eqref{eqn:intyeps} is included in $B_t^\epsilon$ in~\eqref{eqn:sdeas} while the $\sqrt{2}$ is included in $\sigma_\epsilon$.    
\end{Remark}

\tikzset{
    circ/.pic={ 
    \fill (0,0) circle (2pt);
    }
}
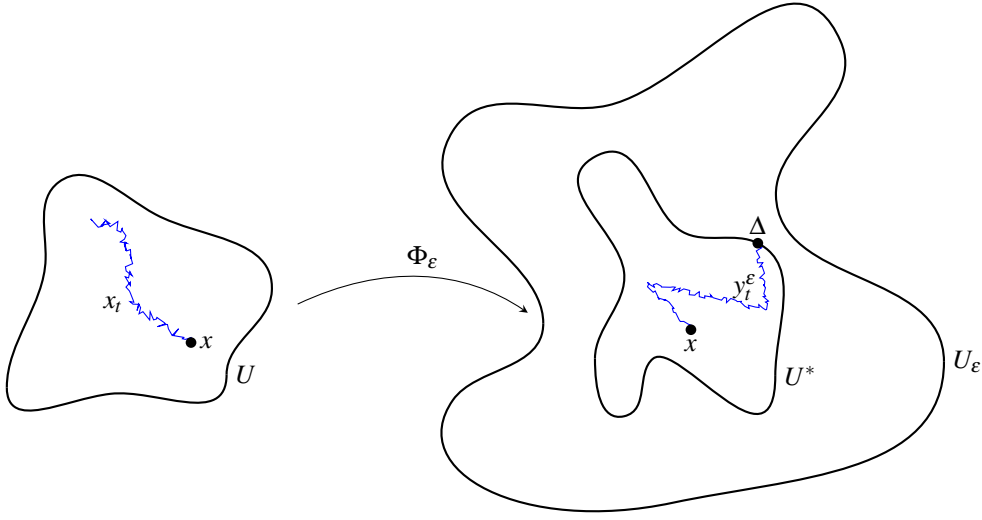
\begin{figure}
\label{fig:1}
\begin{tikzpicture} 
\def\d{1.7};
\def\ddd{0.5*\d}
\coordinate (a1) at (0.7*\ddd,0*\ddd);
\coordinate (a2) at (2.3*\ddd,-0.1*\ddd);
\coordinate (a3) at (4.1*\ddd,0.2*\ddd);
\coordinate (a4) at (4.8*\ddd,1.5*\ddd);
\coordinate (a5) at (3*\ddd,2.7*\ddd);
\coordinate (a6) at (1.5*\ddd,3.2*\ddd);
\coordinate (a7) at (1.3*\ddd,2*\ddd);

\draw[use Hobby shortcut,thick] ([out angle=-90]a1)..(a2)..([in angle=-90]a3); 
\draw[use Hobby shortcut,thick] ([out angle=90]a3)..(a4)..(a5)..(a6)..(a7)..([in angle=90]a1);

\coordinate (S) at (2.5*\d,0);

\coordinate (b1) at (2*\d,0.5*\d);
\coordinate (b2) at (1.3*\d,-0.1*\d);
\coordinate (b11) at (3*\d,-0.9*\d);
\coordinate (b3) at (5.1*\d,0.2*\d);
\coordinate (b4) at (3.8*\d,1.5*\d);
\coordinate (b5) at (4.3*\d,2.7*\d);
\coordinate (b6) at (2.5*\d,2.2*\d);
\coordinate (b7) at (1.3*\d,2*\d);

\draw[use Hobby shortcut,thick] ([out angle=-90]$(b1)+(S)$)..($(b2)+(S)$)..($(b11)+(S)$)..([in angle=-90]$(b3)+(S)$); 
\draw[use Hobby shortcut,thick] ([out angle=90]$(b3)+(S)$)..($(b4)+(S)$)..($(b5)+(S)$)..($(b6)+(S)$)..($(b7)+(S)$)..([in angle=90]$(b1)+(S)$);

\coordinate (Ss) at (4*\d,0);

\def\dd{.45*\d};

\coordinate (c1) at (2*\dd,0.5*\dd);
\coordinate (c2) at (2.5*\dd,-0.5*\dd);
\coordinate (c11) at (3*\dd,.5*\dd);
\coordinate (c3) at (5.1*\dd,0.2*\dd);
\coordinate (c4) at (4.8*\dd,2.5*\dd);
\coordinate (c5) at (3.5*\dd,2.7*\dd);
\coordinate (c6) at (2*\dd,4*\dd);
\coordinate (c7) at (2.5*\dd,2*\dd);

\draw[use Hobby shortcut,thick] ([out angle=-90]$(c1)+(Ss)$)..($(c2)+(Ss)$)..($(c11)+(Ss)$)..([in angle=-90]$(c3)+(Ss)$); 
\draw[use Hobby shortcut,thick] ([out angle=90]$(c3)+(Ss)$)..($(c4)+(Ss)$)..($(c5)+(Ss)$)..($(c6)+(Ss)$)..($(c7)+(Ss)$)..([in angle=90]$(c1)+(Ss)$);

d1 = (c1)!0.5!(c2);

\pgfmathsetseed{2236}

\draw [decorate, decoration={random steps,segment length=1pt,amplitude=2pt}] [blue] plot [tension=1] coordinates { ($(a2)!0.5!(a4)$) ($(a2)!0.5!(a5)$) ($(a5)!0.3!(a7)$) ($(a3)!0.8!(a6)$)};

\coordinate (d1) at ($(c2)!0.5!(c4)$);
\coordinate (d2) at ($(c1)!0.6!(c5)$);
\coordinate (d3) at ($(c3)!0.5!(c4)$);
\coordinate (d4) at ($(c3)!0.5!(c1)$);

\draw [decorate, decoration={random steps,segment length=1.1pt,amplitude=1.2pt}] [blue] plot [tension=1] coordinates { ($(d1)+(Ss)$) ($(d2)+(Ss)$) ($(d3)+(Ss)$) ($(c4)+(Ss)$)};

\draw (a3) node[right] {$U$};
\draw ($(c3)+(Ss)$) node[right] {$U^*$};
\draw ($(b3)+(S)$) node[right] {$U_\epsilon$};

\draw ($(a2)!0.5!(a4)$) node[right] {$x$} pic{circ};
\draw ($(a2)!0.5!(a5)$) node[left] {$x_t$};
\draw ($(d3)+(Ss)$) node[above left] {$y_t^\epsilon$};

\draw ($(d1)+(Ss)$) node[below] {$x$} pic{circ};
\draw ($(c4)+(Ss)$) node[above] {$\Delta$} pic{circ};

\coordinate (e1) at ($(b1)+(S)$);
\coordinate (e2) at ($(a4)+(0.2*\d, -0.1*\d)$);

\draw[-stealth,shorten >= 7pt] (e2.north) to[bend left] node[midway,above] {$\Phi_\epsilon$} (e1.north);

\end{tikzpicture}

\caption{A diagram representing a cartoon relationship between the process $x_t$ and $y_t^\epsilon$ via the generic mapping $\Phi_\epsilon$.  $\Phi_\epsilon$ should be thought of as a rescaling of the equation~\eqref{eqn:sdemain}.  Using the rescaling $\Phi_\epsilon$, typically the domain $U_\epsilon$ grows as $\epsilon \downarrow 0$.}
\end{figure}

     \subsection{The space of explosive trajectories $\EE([0,t]; V)$}
In order to treat possible  finite-time blow-up of either $y^\epsilon$ in~\eqref{eqn:sdeas} 
or $g_t$ (see \eqref{eqn:det} below), we work in the space of \emph{explosive trajectories}.  That is, for any $m\in \N$, $t\in (0, 1]$ and open $V\subset \RR^m$, we again, offering a slight abuse of notation, fix a death state $\Delta \notin V$ and let $V \cup \{ \Delta\}$ be the Alexandroff compactification of $V$.  Then,  $\EE([0,t];V)$ denotes the space of continuous mappings 
\begin{align}
g:[0,t]\rightarrow V\cup \{ \Delta \}
\end{align}
 such that if $g_{t_0}=\Delta$ for some $t_0\in [0,t]$, then $g_s= \Delta$ for all $s\in [t_0,t]$.  For any $x\in V$ and $t\in (0,1]$, we define $\EE_x([0,t];V) := \{ g\in \EE([0,t];V)\,: \, g_0 =x\}$.  If $g\in \EE([0,t];V)$, let
 \begin{align}\label{eqn:bup}
 \tau_t(g)=\inf\{ s \in [0, t] \, : \, g_s=\Delta\}
 \end{align}
denote the \emph{time of explosion} of $g$, where we set $\inf \,\emptyset = \infty$. Define 
\begin{equation}
\EE^0([0, t];V) = \{g\in\EE([0,t]; V) : \dot{g}\in L^2([0, s]) \textrm{ for any } 
s <  \tau_t(g), s \leq t \} \,,
\end{equation} 
 where $\dot{g}$ denotes the time derivative of $g$.  In other words, $g\in \EE^0([0, t];V)$ means that $g$ belongs locally
 to the Sobolev space $H^1([0, \tau_t(g)) \cap[0, t]; V)$. 
We denote by
$\C([0,t];V)$ and $\C_x([0,t]; V)$, $x\in V$, $t\in (0,1]$, respectively  the space of continuous $g:[0,t]\rightarrow V$ and continuous $g:[0,1]\rightarrow V$ with $g_0=x$.  
In particular, $g \in \C([0,t];V)$ implies $\tau_t(g) = \infty$. 
Let 
\begin{align}
\C^0([0,t];V)= \C([0,t];V) \cap \EE^0([0,t];V)
\end{align}  
and observe that $\C^0([0,t];V)$ coincides with the Sobolev space $H^1([0, t], V)$. 

It is important to equip the space $\EE([0,t]; V)$ with a topology compatible with the topology 
of $\C([0,t]; V)$.  As in~\cite{Az_80}, we define the closed sets in $\EE([0,t]; V)$ by specifying convergent
sequences.  That is, we say that a sequence $g_n \in \EE([0,t];V)$ \emph{converges} to $g\in \EE([0,t]; V)$ as $n
\rightarrow \infty$ if $g_n$ converges uniformly to $g$ on compact subsets of $[0, \tau_t(g)) \cap [0, t]$, or equivalently, 
 if for any $s\in [0, \tau_t(g))$, $s \leq t$  there exists $N\in \N$ such that 
$\{g_n\}_{n\geq N} \subset \C([0, s]; V)$ 
 and  $g_n \to g$  as $n \to \infty$ in the space  $\C([0,s];V)$.  The topologies on $\C^0([0,t];V)$ and $\EE^0([0,t];V)$ are then induced by the topologies on, respectively, $\C([0,t]; V)$ and $\EE([0,t]; V)$ intersected with $H^1_{\textrm{loc}}([0, \tau_t(g))\cap[0, t]; V)$.

For $g,h \in \EE([0,t];V)$ and $s\leq t$, we define
\begin{equation}\label{eqn:dfmd}
 d_{s}(g,h) =  \begin{cases}
 \sup_{u\in [0,s]} |g_u-h_u|\,, & \text{ if } s< \tau_t(g) \wedge \tau_t(h) \\
 \infty \, & \text{ otherwise}
  \end{cases}.
\end{equation}
Clearly, $d_s(g, g) = 0$ if $s< \tau_t(g)$ and $d_s(g, h) = d_s(h, g)$, where both sides are either infinite or finite and equal.  Also note that for any $f,g,h \in \EE([0,t];V)$ it follows that
\begin{equation}\label{eqn:trin}
 d_{s}(f,g) \leq  d_{s}(f, h) +  d_{s}(h, g)
\end{equation}
since the above reduces to the usual triangle inequality in $\C([0,s];V)$ if $s< \tau_t(f) \wedge \tau_t(g) \wedge \tau_t(h)$.  On the other other hand if $s\geq \tau_t(f) \wedge \tau_t(g) \wedge \tau_t(h)$, then the righthand side is always infinite, so the inequality~\eqref{eqn:trin} is trivially satisfied. 

\begin{Remark}
Note that $d_s$ is not a metric on $\EE([0,t];V)$ since, for example, $d_s(g,h)=\infty$ if $g,h$ are identically equal to $\Delta$.  However, we do not need $d$ to be a metric below. 
\end{Remark}

 For any $A\subset \C([0,t]; V)$ and any $g \in \EE([0,t]; V)$ we set 
 \begin{align*}
 d_t(g, A) = \inf\{ d_t(g, h) \, : \, h\in A\}.  
 \end{align*}

\begin{Remark}
If $g\in \EE([0,1]; U)$, then its restriction to $[0, t]$, $t\in (0,1]$, also belongs to $\EE([0,t]; U)$.  Below, we slightly abuse notation by denoting this restriction $g$ as well.     
\end{Remark}

\subsection{Statement of the main general results}
To employ a large deviation principle for $y_t^\epsilon$ solving~\eqref{eqn:sdeas}, the following
 family of deterministic ODEs on the open set $U^*$ is of particular importance:
\begin{align}
\label{eqn:det}
\begin{cases}
\dot{g}_t = b(g_t) + \sigma(g_t) \dot{f}_t \,, \\
g_0=x \,,
\end{cases}
\end{align}
where $b, \sigma, U^*$ are as in Assumption~\ref{assump:1}, $f\in \C^0([0,1]; \RR^k)$, and $x\in U^*$. By classical results for equations with locally Lipschitz coefficients (see, for example,~\cite[Proposition 2.3, p. 75]{Az_80}), if $\tau(g)$ denotes the first exit time of the solution of~\eqref{eqn:det} from $U^*$,  
 Assumption~\ref{assump:1} ensures that 
 for any $f\in \C^0([0,1];\RR^k)$ and $x\in \ustar$, the equation~\eqref{eqn:det} has a unique solution $g \in \EE^0([0,1];\ustar)$ provided we set $g(t) = \Delta$ for any $t \geq \tau(g)$. We also assume that \eqref{eqn:det} is always satisfied for any $t \geq \tau(g)$ and we 
 let  $g=S_x(f)$.  Here, $S:\ustar\times \C^0([0,1];\RR^k)\rightarrow \EE^0([0,1];\ustar)$ is 
 the mapping given by 
 \begin{align} \label{eqn:cnt}
 (x, f) \mapsto g = S_x(f) \,.
 \end{align}
In general, the mapping $S$ is not continuous. 
However, for every $a\geq 0$ the restriction of $S$ to $\ustar\times \C_a$ where
 \begin{align}\label{eqn:dca}
 \C_a:= \{ f\in \C^0([0,1];\RR^k) \, : \, \textstyle{\int}_0^1 |\dot{f}_t|^2 \, dt \leq a\}\,,
 \end{align}
 is continuous~\cite[Proposition~2.8, p.75]{Az_80}. Furthermore, by standard arguments (considering the equation for the difference of solutions and estimating using Gronwall's inequality) 
Assumption~\ref{assump:1}  implies that 
for any compact sets 
$K,L\subset \ustar$ with $K\subset \inter(L)$ and for any $a>0$, there exist constants 
$T>0, C>0$ such that
  \begin{itemize}
  \item[(p1)]  For all $x \in K$ and all $f\in \C_a$, $g=S_x(f)$ has $\tau_1(g) >T$ and $g([0, T])\subset L$.  
  \item[(p2)]  For all $x,y \in K, f\in \C_a$, and $s\leq T$
  \begin{align*}
  d_{s}(S_x(f), S_y(f)) \leq e^{Cs} |x-y|.  
  \end{align*}
  \item[(p3)]  If $\{f_n\} \subset \C_a$ converges in $H^1([0,1])$ to $f\in \C_a$, then  $S_x(f_n) \to S_x(f)$ in $\C^0([0, T];\ustar)$. 
 \end{itemize}  

An essential notion for us is the \emph{Cramer transform} $\lambda: \EE([0,1]; \ustar)\rightarrow [0, \infty]$
defined by
  \begin{align}
  \label{def:lambda}
  &\lambda(g) = \inf\{\tfrac{1}{2} \textstyle{\int_0^1} |\dot{f}_t|^2 \, dt \, : \, f\in \C^0([0,1]; \RR^k) \text{ and } S_{g_0}(f)=g \} \,, 
  \end{align}
  where we set $\inf\, \emptyset = \infty$ and $\dot{f}_s = 0$ for any $s \geq \tau_1(g)$, or equivalently $S_x(f)(s)=g_s$ is satisfied for any $f$ if $s \geq \tau_1(g)$. 
We remark that we can choose $\dot{f}_t = 0$ for $t > \tau_1(g)$ since in this time range the equation \eqref{eqn:det} is satisfied by definition for any $f$.   
  Note that 
  our definition is equivalent to one introduced in \cite[(4), Chapter IV]{Az_80} by \cite[Proposition 2.10, Chapter IV]{Az_80}.  Now \cite[Proposition 2.10, Chapter IV]{Az_80} yields that 
   if $\lambda(g) < \infty$, 
  then the infimum in \eqref{def:lambda} is attained, $g\mapsto \lambda(g)$ is lower semi-continuous, and 
 for any compact set $K \subset \ustar$ and any $a\geq 0$, the set 
  \begin{align}
  \label{eqn:Kcrosslambda}
\{ g\in \EE([0,1];\ustar)\, : \, g_0 \in K, \, \lambda(g) \leq a\}
  \end{align}
  is compact in $\EE([0,1];\ustar)$.  It thus follows that for every $x\in \ustar$ the set 
\begin{align}
\K_x = \{ g\in \EE_x([0,1]; \RR^d) \, : \, \lambda(g) \leq 1 \} 
\end{align}
is also compact by choosing $K=\{ x\}$ and $a=1$ in~\eqref{eqn:Kcrosslambda}. 

By property (p1) and the relation~\eqref{def:lambda}, for any compact set $L \subset \ustar$ and $x \in \inter(L)$ 
there exists a time
\begin{align}
\label{eqn:tdef}
t_*=t_*(x)\in (0,1]
\end{align}
 such that for all $g\in \K_x$, $g([0, t_*]) \subset \inter(L)$. Let $\K_x(t_*)$ be the set of functions in $\K_x$ restricted to $\EE([0,t_*]; \ustar)$.  Our goal is to show that, under Assumption~\ref{assump:1} and Assumption~\ref{assump:2} (defined below) $\K_x(t_*)$ is the set of limit points, almost surely, of $y^\epsilon$ solving~\eqref{eqn:sdeas} as $\epsilon \rightarrow 0$ (when restricted to $[0, t_*]$).  
 Then, we show how one can relate $y^\epsilon$ back to the original process $x_t$ solving~\eqref{eqn:sdemain} to obtain the desired functional LIL.

\begin{Assumption}
\label{assump:2}
Consider the compact set $L \subset \ustar$ and constant $t_*>0$ defined in~\eqref{eqn:tdef}.  The process $y^\epsilon$ with $x \in \inter(L)$ solving~\eqref{eqn:sdeas} satisfies the following properties:
\begin{itemize}
\item[(i)]  For every $\delta >0$, there exists $c_0 \in (0, 1)$ such that for any $c\in (c_0, 1)$ there is a $\PP$-almost surely finite random variable $J = J(\omega, c)\in \N$ such that $j\geq J$ and $\epsilon \in [c^{j+1}, c^j]$ implies 
$y_{t}^\epsilon \in L$ for all $t\in [0, t_*]$ and 
\begin{align}\label{eqn:epj}
d_{t_*}( y^{c^j}, y^{\epsilon}) < \delta.  
\end{align}  
\item[(ii)] For $\epsilon_* > 0$ as in Assumption \ref{assump:1}, 
the mapping $\epsilon \mapsto y^\epsilon : (0, \epsilon_*)\rightarrow \EE_x ([0, t_*];\ustar) $ is continuous, $\PP$-almost surely.  

\end{itemize} 
\end{Assumption}

\begin{Remark}
Assumption~\ref{assump:2} essentially allows one to reduce the proof of the main result (Theorem~\ref{thm:main} below) to the countable sequence $\{ y^{c^{j}}\}$ instead of $\{ y^\epsilon \}$.    
\end{Remark}

\begin{Theorem}
\label{thm:main}
Suppose that Assumption~\ref{assump:1} and Assumption~\ref{assump:2} are both satisfied for some $\epsilon_* > 0$ and non-empty open $\ustar$. Fix compact $L \subset \ustar$,  $x\in \emph{\inter} (L)$ and 
$t_*=t_*(x)\in (0,1]$ as in~\eqref{eqn:tdef}. 
Then, for $\PP$-almost every $\omega$, we have the following conclusions:
\begin{itemize}
\item[(i)]  The set $\Y(\omega) := \{y^\epsilon(\omega)\}_{\epsilon \in (0, \epsilon_*]}$ is relatively compact in $\EE([0, t_*]; \ustar)$.
\item[(ii)]  $d_{t_*}(y^\epsilon(\omega), \K_x(t_*))\rightarrow 0$ as $\epsilon \rightarrow 0$.
\item[(iii)]  For every $h \in \K_x(t_*)$, $\{y^{\epsilon}(\omega)\}_{ \epsilon \in (0, \epsilon_*]}$ has a subsequence $\{y^{\epsilon_j(\omega,h) }(\omega)\}_{j=1}^\infty$ with $\epsilon_j(\omega,h) \downarrow 0$ as $j\rightarrow \infty$ such that 
\begin{align*}
d_{t_*}(y^{\epsilon_j (\omega,h)}(\omega), h)\rightarrow 0 \,\,\text{ as } \,\, j\rightarrow \infty.   
\end{align*}
\end{itemize}
\end{Theorem}

In Section~\ref{sec:examples} we make heavy use of the following corollary of Theorem~\ref{thm:main}, which 
 is a basic topological consequence of relative compactness and continuity.       
 
\begin{Corollary}
\label{cor:1}
Under the hypotheses of Theorem~\ref{thm:main}, let $X$ be a Hausdorff topological space and suppose that $F:\EE([0, t_*]; \ustar) \rightarrow X$ is continuous.  Then for $\PP$-almost every $\omega$, $F(\Y (\omega))$ is relatively compact in $X$ and $F(\K_x(t_*))$ is the limit set of $F(\Y(\omega))$ 
as $\epsilon \to 0$.    
\end{Corollary}

\begin{proof}[Proof of Corollary~\ref{cor:1}]
The second conclusion follows immediately by continuity.  Since $\Y(\omega)$ is relatively compact almost surely, then $\overline{\Y(\omega)}$ is compact, almost surely.  
For $\PP$-almost every $\omega$, $F(\overline{\Y(\omega)})$ is compact since $F$ is continuous, and therefore closed in the Hausdorff topological space $X$. Hence, 
$F(\Y(\omega)) \subset F(\overline{\Y(\omega)})$ implies $\overline{F(\Y(\omega))} \subset F(\overline{\Y(\omega)})$.  Since a closed subset of a compact set is compact,  
the assertion follows. 
\end{proof}

In practice, it is relatively straightforward to check Assumption~\ref{assump:1}.  However, Assumption~\ref{assump:2}, especially part (i), requires more work to validate.  We next explore verifiable conditions under which Assumption~\ref{assump:2} holds.       

\subsection{Sufficient conditions for Assumption~\ref{assump:2}}

Heuristically, if Assumption~\ref{assump:1} is satisfied and the system~\eqref{eqn:sdeas} arises from the original equation~\eqref{eqn:sdemain} under a \emph{reasonable} change of coordinates, then Assumption~\ref{assump:2}  also holds.  However, even more is true if the noise is additive.  That is, if $\sigma_\epsilon(x)\equiv \sigma_\epsilon$ is a constant matrix for every $\epsilon \in (0, \epsilon_*]$ and Assumption~\ref{assump:1} holds, then with an additional marginal continuity hypothesis,  Assumption~\ref{assump:2} holds for ~\eqref{eqn:sdeas} independent of any relationship to the original equation~\eqref{eqn:sdemain}.

To introduce an allowable change of coordinates that maps~\eqref{eqn:sdemain} to~\eqref{eqn:sdeas}, we need further notation.  We denote by $\alpha= (\alpha_1, \ldots, \alpha_d)$ and $\beta = (\beta_1, \ldots, \beta_d)$ 
multiindices taking values in $\RR^d$, and we write $\alpha \succeq \beta $ (respectively $\alpha \succ \beta$) if $\alpha_i \geq \beta_i$ (respectively $\alpha_i > \beta_i$) for all $i$.  
Note that this is equivalent to the partial ordering on the positive cone.  When the context is clear, we use $0$ and $1$ to denote the multiindices $(0,0,\ldots, 0)$ and $(1,1, \ldots, 1)$, respectively.  For multiindices $\alpha, \beta$, we let the product $\alpha \beta$ denote the multiindex $(\alpha_1 \beta_1, \ldots, \alpha_d \beta_d)$ and if $\alpha \succ 0$,  we define the multiindex $\alpha^{-1}=(\alpha_1^{-1}, \alpha_2^{-1}, \ldots, \alpha_d^{-1})$.  Finally, for any multi-index $\alpha$, we 
define 
\begin{equation}
|\alpha| = \sqrt{\sum_i \alpha_i^2}\,. 
\end{equation}

\begin{Definition}
\label{def:1}
Suppose that, for every multiindex $\alpha \succ 0$, $\Phi_\alpha: U \rightarrow U_\alpha$ is a $C^2$-bijection.  We call $\{ \Phi_\alpha \}_{\alpha \succ 0}$ \emph{a family of weak contractions centered at 
$x\in U$} if the following conditions are met: 
\begin{itemize}
\item[(i)] For every multiindex $\alpha \succ 0$, $\Phi_\alpha (x)= x$;
\item[(ii)] For all multiindices $\alpha \succeq \beta\succ 0$ we have
\begin{align*}
|\Phi_{\alpha}(y) - \Phi_{\alpha}(z) | \leq |\Phi_{\beta}(y) - \Phi_{\beta}(z) |
\end{align*}
for all $y,z \in U$.  
\item[(iii)]  There exist $\kappa > 0$ and an open set $\ustar \subset \RR^d$ such that $x \in \ustar$, 
$\ustar \subset U_\alpha$ for each $|\alpha| < \kappa$, and for any compact set $K\subset \ustar$ and any $\epsilon >0$ there exists $\delta >0$ such that  $|\alpha \beta^{-1} -1 |<\delta$ implies 
\begin{align*}
| \Phi_\alpha \circ \Phi_\beta^{-1} (y) - y | < \epsilon
\end{align*}
for all $y\in K$.  
\end{itemize}
 \end{Definition}

\begin{Example}
\label{ex:1}
Let $U$ be an open neighborhood of $0\in \RR^d$, and for any multiindex  $\alpha \succ 0$, let $\Phi_\alpha^1: U\rightarrow \RR^d$ be given by 
\begin{align}
\Phi_\alpha^1(y):= (y_1 \alpha_1^{-1}, y_2 \alpha_2^{-1}, \ldots, y_d \alpha_d^{-1}).
\end{align}
Define the open set $U_\alpha= \Phi_\alpha^1(U)$.  Then, $\{ \Phi_\alpha^1 \}_{\alpha \succ 0 }$ defines a family of weak contractions centered at $0\in U$.  
Also, by \emph{shifting} everything above, if $U$ now denotes an open neighborhood of $x\in \RR^d$, the family $\{ \Phi_\alpha\}_{\alpha \succ 0}$ defined by 
\begin{align*}
\Phi_\alpha(y) = x- \Phi_\alpha^1(x)+ \Phi_{\alpha}^1(y) = x + \Phi_\alpha^1(y - x), \,\,\,\, y\in U,   
\end{align*}
is a family of weak contractions centered at $x$ by setting $U_\alpha=\Phi_\alpha(U)$. 
\end{Example}

\begin{Remark}
Compared with Baldi~\cite{Bal_86} and Caramellino~\cite{Car_98}, the index $\alpha$ in Definition \ref{def:1} 
is allowed to be a multiindex rather than a positive real parameter. 
 With some minor additional structure (see Definition~\ref{def:2} below), this affords more general transformations of~\eqref{eqn:sdemain} rather than functions of $\sqrt{\epsilon \log \log \epsilon^{-1}}$ alone (see Section~\ref{sec:examples}).  It is expected that a similar condition can be used to deduce LILs for diffusions at time infinity as well.        
\end{Remark}

In order to specify a change of coordinates from $x_t$ to $y^\epsilon_t$, we need to impose assumptions on the dependence of the multiindex $\alpha$ 
on $\epsilon$.  Below, this dependence is determined using a heuristic scaling argument which in turn dictates the asymptotic behavior
of $x_t$ at time $t = 0$. 
The conditions outlined in the next definition are natural and satisfied in the examples in which we are interested. 

\begin{Definition}
\label{def:2}
Fix $\epsilon_0 > 0$. 
We call $\psi:[0, \epsilon_0] \rightarrow [0, \infty)^d$ an \emph{asymptotic index} if
all of the following conditions are met:
 \begin{itemize}
\item[(i)]  $\psi$ is continuous.
\item[(ii)]   $\psi(0)=0$ and $\psi(u)\prec \psi(v)$ as a multiindex for any $0 < u < v\leq \epsilon_0$. 
\item[(iii)] For any $\epsilon >0$, there exists $\delta \in (0,1)$ such that for all $c\in (1-\delta, 1)$ there exists $J\in \N$ such that for any $j\geq J$:
\begin{align*}
\delta_1, \delta_2 \in [ c^{j+1}, c^j]\,\, \text{ implies }\,\, | \psi(\delta_1) \psi(\delta_2)^{-1} -1| < \epsilon.  
\end{align*}
\end{itemize}
\end{Definition}

\begin{Example}
\label{ex:2}
For any positive integers $\ell, k\in \N$, and $\epsilon_*(\ell, k) >0$ small enough, let $\psi_{\ell, k}:[0, \epsilon_*(\ell,k)]\rightarrow [0, \infty)$ be given by 
\begin{align}
\label{eqn:functionex}
\psi_{\ell, k}(\epsilon)= \begin{cases}
\sqrt{\epsilon^\ell (\log \log \epsilon^{-1})^k}& \text{ for } \epsilon \in(0, \epsilon_*(\ell, k)] \,,\\
0 & \text{ for } \epsilon =0.
\end{cases}
\end{align}
  Then, for any $(\ell_1, k_1), \ldots, (\ell_d, k_d)\in \N \times \N$ and $\epsilon_*:=\min\{\epsilon_*(\ell_i, k_i) \, : \, i =1,\ldots, d\}>0$, $\psi:[0, \epsilon_*]\rightarrow [0, \infty)^d$ defined as 
\begin{align*}
\psi(\epsilon)= (\psi_{\ell_1, k_1}(\epsilon), \ldots, \psi_{\ell_d, k_d}(\epsilon) )
\end{align*}    
is an asymptotic index.  Indeed, by standard calculations for $\epsilon \in (0, \epsilon_*]$ and $\epsilon_*>0$ small enough we have 
\begin{equation}
\frac{d}{d\epsilon}(\psi_{\ell,k} (\epsilon))^2 = \epsilon^{\ell - 1}(\log\log \epsilon^{-1})^{k - 1} \left(\ell \log\log \epsilon^{-1} - \frac{k}{\log \epsilon^{-1}} \right) > 0 \,,
\end{equation}
which implies (ii). Also,  by just proved monotonicity, if $c^{j + 1} \leq \delta_2 \leq \delta_1 \leq c^j$
\begin{equation}
1 \leq 
\frac{\psi_{\ell, k}(\delta_1)}{\psi_{\ell, k}(\delta_2)} = \left(\frac{\delta_1}{\delta_2}\right)^{\frac{\ell}{2}} \left(\frac{\log\log \delta_1^{-1}}{\log\log \delta_2^{-1}}\right)^{\frac{k}{2}}
\leq \frac{1}{c^{\frac{\ell}{2}}} \left(\frac{\log (j\log c^{-1})}{\log ((j + 1)\log c^{-1})}\right)^{\frac{k}{2}} \,.
\end{equation}
Then, for any small $\epsilon > 0$ there exists $\delta > 0$ such that 
$c^{\frac{\ell}{2}} > 1 - \frac{\epsilon}{3}$ for any $c \in (1-\delta, 1)$ and by choosing $J$ large, we obtain that for any $j \geq J$ 
\begin{equation}
\left(\frac{\log (j\log c^{-1})}{\log ((j + 1)\log c^{-1})}\right)^{\frac{k}{2}} < 1 + \frac{\epsilon}{3}.
\end{equation}
Thus, if $\epsilon \in (0,1)$
\begin{align}
0 \leq \frac{\psi_{\ell, k}(\delta_1)}{\psi_{\ell, k}(\delta_2)} - 1 \leq \frac{1 + \frac{\epsilon}{3}}{1 - \frac{\epsilon}{3}} - 1 < \epsilon.
\end{align}
Hence property (iii) follows. The case $c^{j + 1} \leq \delta_1 \leq \delta_2 \leq c^j$ follows analogously. 
\end{Example}

\begin{Example}
\label{ex:3}
In Example~\ref{ex:2}, instead of choosing each coordinate of the form~\eqref{eqn:functionex}, one could replace $\psi_{\ell, k}$ by a continuous, strictly increasing function $\varphi:[0, \epsilon_*]\rightarrow [0, \infty)$ with $\varphi(0)=0$ and $\varphi$ regularly varying at $0$.     
\end{Example}

Using the previous two concepts, we now connect equations~\eqref{eqn:sdemain} and~\eqref{eqn:sdeas}.  Suppose that $\{ \Phi_\alpha\}_{\alpha \succ0}$ is a family of weak contractions centered at 
$x\in U$ and $\psi:[0, \epsilon_*]\rightarrow \RR^d$ is an asymptotic index.  Observe that for any $\epsilon \in (0,\epsilon_*]$ and $t< \epsilon^{-1} \tau(x_\cdot)$ (see \eqref{eqn:bup}), the family of processes 
\begin{align}
\label{eqn:coc}
y^\epsilon_t:= \Phi_{\psi(\epsilon)}(x_{\epsilon t}) 
\end{align}
satisfies, by It\^ o's formula, an SDE on $U_\epsilon$ of the form~\eqref{eqn:sdeas} with $b_\epsilon:U_\epsilon\rightarrow \RR^d, \sigma_\epsilon:U_\epsilon\rightarrow M_{d\times k}$ given by 
\begin{align}
\label{eqn:beps1}
b_\epsilon(y) &= \epsilon \tilde{L} \Phi_{\psi(\epsilon)}( \Phi^{-1}_{\psi(\epsilon)}(y)), \\
\label{eqn:sigeps1}  \sigma_\epsilon(y)& = \sqrt{\epsilon r(\epsilon)} D \Phi_{\psi(\epsilon)}(\Phi_{\psi(\epsilon)}^{-1}(y))  \tilde{\sigma}(\Phi^{-1}_{\psi(\epsilon)}(y)))\,,
\end{align}
where 
 \begin{align}
 \label{eqn:gen}
 \tilde{L} = \sum_{i=1}^d \tilde{b}_i(x) \frac{\partial}{\partial x_i} + \frac{1}{2}\sum_{i,j=1}^d (\tilde{\sigma}(x) \tilde{\sigma}(x)^T)_{ij} \frac{\partial^2}{\partial x_i \partial x_j}  
 \end{align}
 with $\tilde{b}$ and $\tilde{\sigma}$ as in~\eqref{eqn:sdemain}. 
 \begin{Lemma}
 \label{lem:1}
 Suppose that $\{ \Phi_\alpha\}_{\alpha \succ0}$ is a family of weak contractions centered at $x\in U$ and $\psi:[0, \epsilon_*]\rightarrow \RR^d$ is an asymptotic index
and suppose $|\psi(\epsilon_*)| < \kappa$, where $\kappa$ is as in Definition \ref{def:1}(iii) for the appropriate $\ustar$.  
If $b_\epsilon, \sigma_\epsilon$ given by~\eqref{eqn:beps1}-\eqref{eqn:sigeps1} satisfy Assumption~\ref{assump:1} with already fixed $\epsilon_* > 0$ and $\ustar$, 
 then the family of processes $\{y^\epsilon\}_{\epsilon \in (0, \epsilon_*]}$ given by~\eqref{eqn:coc} satisfies Assumption~\ref{assump:2}.  
  \end{Lemma}
Lemma~\ref{lem:1} is proved in Section~\ref{sec:lem} along with the following corollary.

\begin{Corollary}
\label{cor:2}
Consider the family of processes $\{y^\epsilon\}_{\epsilon \in (0, \epsilon_*]}$ defined by relation~\eqref{eqn:sdeas} and suppose that $\sigma_\epsilon(x)\equiv \sigma_\epsilon$ is a family of constant $d\times k$ matrices and the Brownian motion $B_t^\epsilon$ in~\eqref{eqn:sdeas} is given by $B^\epsilon_t:=\tfrac{1}{\sqrt{\epsilon}}B_{\epsilon t}$.  
If $\epsilon_*$, $\ustar$, $b_\epsilon, \sigma_\epsilon$ satisfy Assumption~\ref{assump:1} and for every $\delta \in (0, \epsilon_*]$ we have that $\sigma_\epsilon \rightarrow \sigma_\delta$ and $b_\epsilon \rightarrow b_\delta$ as $\epsilon \rightarrow \delta$ uniformly on compact subsets of 
$\ustar$, then  $\{y^\epsilon\}_{\epsilon \in (0, \epsilon_*]}$ satisfies Assumption~\ref{assump:2}.  
\end{Corollary}

\begin{Remark}
Perhaps the most surprising consequence of Lemma~\ref{lem:1} is that it is used to prove Corollary~\ref{cor:2}, even though there is no reference to an underlying mapping from $x_t$ to $y_t^\epsilon$ . 
\end{Remark}

\section{Law of the Iterated Logarithm Examples in the weakly hypoelliptic setting}\label{sec:examples}

\subsection{Weakly hypoelliptic diffusions}
Since Theorem~\ref{thm:main} is a generalization of the main result in~\cite{Car_98}, all applications discussed there also follow from Theorem~\ref{thm:main}.  We therefore refer the reader to~\cite{Car_98} to see how to obtain a functional LIL at time zero for $d$-dimensional Brownian motion, elliptic SDEs as well as some iterated stochastic integrals.  Here, we provide examples not covered by~\cite{Car_98} which follow from Theorem~\ref{thm:main}.

All of the SDEs discussed below fall within the class of \emph{weakly hypoelliptic} diffusions with additive noise; that is, each SDE below is of the form~\eqref{eqn:sdemain},
 where $\tilde{b} \in C^\infty(U)$ and $\tilde{\sigma}(x) \equiv \tilde{\sigma}$ is a $d\times k$ constant matrix  such that the range of $\tilde{\sigma}$, denoted by $\mathcal{R}(\tilde{\sigma})$, has dimension strictly less than $d$, but \emph{H\"{o}rmander's condition} is satisfied.  That is,   
we say that the the columns  $\tilde{\sigma}^1, \tilde{\sigma}^2, \ldots, \tilde{\sigma}^k$ of $\tilde{\sigma}$ and  $\tilde{\sigma}^0(x):=\tilde{b}(x)$, viewed as vector fields on $U$, satisfy \emph{H\" ormander's condition} on $U$
if the list
\begin{align}
\tag{H}\label{eqn:WH}
\tilde{\sigma}^{\ell_1}(x) &\qquad \ell_1 =1,2,\ldots, k \\
\nonumber [\tilde{\sigma}^{\ell_1}, \tilde{\sigma}^{\ell_2}](x)&\qquad \ell_1, \ell_2 =0,1,\ldots , k \\
\nonumber [\tilde{\sigma}^{\ell_1}, [\tilde{\sigma}^{\ell_2}, \tilde{\sigma}^{\ell_3}]](x) & \qquad \ell_1, \ell_2, \ell_3= 0,1,\ldots, k\\
\nonumber\vdots & \qquad \vdots 
\end{align}  
spans the tangent space at all points $x\in U$.  In the above, $[X,Y]$ denotes the \emph{commutator} of the vector fields $X$ and $Y$; that is, if $X=(X_j(x))$ and $Y=(Y_j(x))$, then 
\begin{align*}
[X,Y](x) := \sum_{j=1}^d \sum_{i=1}^d \bigg\{ X_i(x) \frac{\partial Y_j(x)}{\partial x_i} - Y_i(x) \frac{\partial X_j(x)}{\partial x_i} \bigg\} \frac{\partial}{\partial x_j}.
\end{align*} 

A celebrated theorem of H\"{o}rmander~\cite{Hor_67} shows that if condition~\eqref{eqn:WH} is satisfied, then the operators $\tilde{L}$, $\tilde{L}^*$, $\partial_t \pm \tilde{L}$, $\partial_t \pm \tilde{L}^*$ are all hypoelliptic on the respective domains $U$, $U$, $(0, \infty) \times U$, $(0, \infty) \times U$, where 
$\tilde{L}$ is as in~\eqref{eqn:gen} and $\tilde{L}^*$ denotes the formal $L^2(dx)$-adjoint of $\tilde{L}$.  
As a consequence, the distribution of the solution process $x_t$ restricted to Borel subsets of $U$ is absolutely continuous with respect to Lebesgue measure with density $q_t(x,y)$.  Futhermore, $(t,x,y) \mapsto q_t(x,y)\in C^\infty ((0, \infty) \times U \times U)$.  If the process $x_t$ exits $U$ in finite time, then the law of $x_t$ has 
a singular component on $\partial U$.  However, this component is not present prior to exiting.

One interpretation of H\"{o}rmander's theorem is that condition~\eqref{eqn:WH} ensures that $x_t$ is not locally restricted to a lower-dimensional submanifold of $U$.  Indeed, 
the noise is either acting explicitly in directions $\tilde{\sigma}^1, \ldots, \tilde{\sigma}^k$, or it propagates implicitly through the drift term $\tilde{\sigma}^0=\tilde{b}$, as represented by the  
commutators in~\eqref{eqn:WH}.  However, condition~\eqref{eqn:WH} does not guarantee that the process reaches all points in a small neighborhood in short times. 
For example, the process may be restricted to a cone (still satisfying H\"{o}rmander condition) as opposed to  a ball~\cite[Example~3.4]{HerMat_15ii}.  
The goal of our LILs deduced below is to provide further insight into the a.s., small time behavior.

\subsection{Examples}  We now consider several concrete examples, starting with the so-called \emph{Iterated Kolmogorov} diffusion.  

\begin{Example}[Iterated Kolmogorov]
\label{ex:IK}
Consider the following SDE on $\RR^d$
\begin{equation}\label{eqn:itkol}
\begin{aligned}
dx_1 &= x_2 \, dt \,,\\
 dx_2&=x_3 \, dt \,,\\
 \vdots & \qquad \vdots \\
 dx_{d-1}&= x_d \, dt \,,\\
 dx_d&= dW_t \,,
\end{aligned}
\end{equation}
where $W_t$ is a standard one-dimensional Brownian motion defined on $(\Omega, \mathcal{F},  \PP)$ and the process $$x_t:=(x_1(t), \ldots, x_d(t))$$ has initial condition $x_0=(0,0,\ldots, 0)$.  Our goal is to establish an LIL for the first coordinate $x_1$, which is given by the iterated time integral of the Brownian motion
\begin{align*}
x_1(t) = \int_0^t \int_0^{t_2} \dots \int_0^{t_{d}} dW_s \,  dt_{d} \, dt_{d-1} \ldots d t_2 .  
\end{align*}  
This was the one of the main goals of the paper~\cite{Lac_97} by Lachal.  Historically, the case  $d=2$ in~\eqref{eqn:itkol} is the first known example of a \emph{hypoelliptic} diffusion, as discovered by Kolmogorov.  For further information, see the discussion in the introduction of~\cite{Hor_67}.    

For any multiindex $\alpha =(\alpha_1, \alpha_2, \ldots, \alpha_d) \succ 0$,  define $\Phi_\alpha: \RR^d \rightarrow \RR^d$ by
\begin{align*}
\Phi_\alpha(y) = (\alpha_1^{-1} y_1, \alpha_2^{-1} y_2, \ldots, \alpha_d^{-1} y_d) 
\end{align*}   
and note by Example~\ref{ex:1}, $\{ \Phi_\alpha \}_{\alpha \succ 0}$ is a family of weak contractions centered at the origin in $ \RR^d$.  Furthermore, by Example~\ref{ex:2},  for $\epsilon_*>0$ small enough, $\psi:[0, \epsilon_*]\rightarrow [0, \infty)^d$ given by 
\begin{align*}
\psi(\epsilon) = \Big( \sqrt{\epsilon^{2d-1} \log \log \epsilon^{-1}}, \sqrt{\epsilon^{2d-3} \log \log \epsilon^{-1}} ,          \ldots ,\sqrt{\epsilon^3 \log \log \epsilon^{-1}}  , \sqrt{\epsilon \log \log \epsilon^{-1}}\Big)
\end{align*}
is an asymptotic index.  To see that Assumption~\ref{assump:1} is satisfied for the transformed diffusion $y_t^\epsilon$ defined by 
\begin{align}
\label{def:yepsik}
y^\epsilon_t = \Phi_{\psi(\epsilon)} (x_{\epsilon t}),  
\end{align}   
we observe that, by construction, $y^\epsilon_t$ solves the following SDE
\begin{equation}
\begin{aligned}
dy_1&= y_2 \, dt  \,, \\ 
 \vdots & \qquad \vdots \\ 
 dy_{d-1}&= y_d \, dt \,, \\
 dy_d &= \frac{dB_\epsilon(t)}{\sqrt{r(\epsilon)}}  \,,
\end{aligned}
\end{equation}
with $y_0=(y_1(0), \ldots, y_d(0))=0$ and $B_\epsilon(t)=\epsilon^{-1/2} W_{\epsilon t}$ being a standard Brownian motion on $\RR$.  Since, $b_\epsilon\equiv b$ is a linear function on $\RR^d$ and $\sigma_\epsilon \equiv \sigma$ is a constant matrix,  Assumption~\ref{assump:1} is satisfied. In addition, by Lemma~\ref{lem:1}, Assumption~\ref{assump:2} holds true for $y^\epsilon_t$.  
Furthermore, since the process is non-explosive, we may set $t_*=1$ in the statements of Theorem~\ref{thm:main} and Corollary~\ref{cor:1}.      

For 
 the projection  $\pi:\RR^d\rightarrow \RR$  onto the first coordinate, consider the continuous map $F:\EE([0, 1]; \RR^d)\rightarrow \RR\cup \{ \Delta\}$ given by $F(g)=\pi g_1$ if $g_1\in \RR^d$ and $F(g) = \Delta$ if $g_1 =\Delta$, where $g_1 = g_{t = 1}$.  Observe that $\RR \cup \{ \Delta \}$ is a Hausdorff space.  Then, Corollary~\ref{cor:1} implies that $\{ F(y^\epsilon) \}_{0< \epsilon < \epsilon_*}$ is relatively compact in $\RR \cup \{ \Delta \}$.  Furthermore, for $f\in \C^0([0,1]; \RR)$ let
\begin{align}
J_1(f)= \int_0^1 \int_0^{t_2} \ldots \int_0^{t_{d-1}} f_{t_d} \, dt_d \ldots dt_2 \,.
\end{align}       
Then, the almost sure limit set  of $(y^\epsilon_1)_1$ as $\epsilon \rightarrow 0$ is given by 
\begin{align*}
F(\K_0)= \{ J_1(f)\, : \,  \textstyle{\frac{1}{2}\int_0^1} (\dot{f}_s)^2 \, ds \leq 1 \}. 
\end{align*}
The embedding $H^1  ([0, 1]) \hookrightarrow L^\infty([0, 1])$ implies that the constants 
\begin{align*}
M&= \sup \{ J_1(f)\, : \,  \textstyle{\frac{1}{2}\int_0^1} (\dot{f}_s)^2 \, ds \leq 1 \},\\
m&= \inf\{ J_1(f) \, : \, \textstyle{\frac{1}{2}\int_0^1} (\dot{f}_s)^2 \, ds \leq 1 \},
\end{align*}
are finite and by choosing $f_s = \pm s$, one has $M, m \neq 0$, and therefore almost surely  
\begin{align*}
\limsup_{\epsilon \rightarrow 0} \frac{x_1(\epsilon)}{ \sqrt{\epsilon^{2d-1} \log \log \epsilon^{-1}}} &=M>0 \,,\\
\liminf_{\epsilon \rightarrow 0} \frac{x_1(\epsilon)}{ \sqrt{\epsilon^{2d-1} \log \log \epsilon^{-1}}} &=m.\end{align*}
Moreover, $m=-M$ since $J$ is an odd function of $f$.

\end{Example}

\begin{Example}
\label{ex:IK2}
This example shows the utility of Corollary~\ref{cor:2}.  
Consider again the same system as in Example~\ref{ex:IK}, but with $d=2$ and the process $x_t=(x_1(t), x_2(t))$ starting from a general initial condition $x_0=(x_1(0), x_2(0)) \in \RR^2$.  
Equivalently, we can consider the process solving 
\begin{equation}\label{eqn:dtk}
dx_1 = (x_2 + c)\, dt, \qquad dx_2 = dW_t
\end{equation}
for some $c$, with $x_0 = 0$. 
Again, our goal is to obtain a LIL for the first coordinate $x_1$.  
If we define
\begin{align}
\label{eqn:cocIK}
y_1^\epsilon(t)&= \frac{x_1(\epsilon t)-x_1(0)- t \epsilon x_2(0)}{\sqrt{\epsilon^3 \log \log \epsilon^{-1}}} + x_1(0) + tx_2(0) \,,\\
\nonumber y_2^\epsilon(t)&= \frac{x_2(t\epsilon)-x_2(0)}{\sqrt{\epsilon \log \log \epsilon^{-1}}} + x_2(0) \,,
\end{align}    
then  $y_t^\epsilon :=(y_1^\epsilon(t), y_2^\epsilon(t))$ satisfies 
\begin{align*}
dy_1^\epsilon&= y_2^\epsilon \, dt \,,\\
dy_2^\epsilon&= \tfrac{1}{\sqrt{r(\epsilon)}} dW^\epsilon_t \,,  
\end{align*}
where $(y_1^\epsilon(0), y_2^\epsilon(0))= (x_1(0), x_2(0))$ and $W^\epsilon_t= \tfrac{1}{\sqrt{\epsilon}} W_{\epsilon t}$.  Note that Assumption~\ref{assump:1} is clearly satisfied for the process $y_t^\epsilon$.  
Due to the explicit dependence on time, the mapping $x \mapsto y^\epsilon$ in \eqref{eqn:cocIK} does not satisfy assumptions of 
Lemma~\ref{lem:1} if $x_2(0)\neq 0$.  However, Corollary~\ref{cor:2} ensures that  Assumption~\ref{assump:2} is satisfied.  Also, 
the solution of the associated deterministic problem 
\begin{equation}
\dot{y}_1 = y_2, \qquad \dot{y}_2 = \dot{f} 
\end{equation}
is given by 
\begin{equation}
y_2(t) = x_2(0) + f(t), \qquad y_1(t) = x_1(0) +  x_2(0)t + \textstyle{\int_0^t f(s) ds} \,,
\end{equation}
where $f(0) = 0$. 
If  $J_1: \C^0([0,1]; \RR)\rightarrow \RR$ is given by 
\begin{align}
J_1(f) = \int_0^1 f_s \, ds  \qquad \text{ and } \qquad M= \sup\{ J_1(f) \, : \, \tfrac{1}{2}\textstyle{\int_0^1 }(\dot{f}_s)^2 \, ds \leq1 \},
\end{align}    
then as in Example \ref{ex:IK}, one has $0 < M < \infty$ and, $\PP$-almost surely,
\begin{align}
\label{eqn:limsupIK}
\limsup_{\epsilon \rightarrow 0}  \bigg(\frac{x_1(\epsilon )-x_1(0)- \epsilon x_2(0)}{\sqrt{\epsilon^3 \log \log \epsilon^{-1}}} \bigg) &= M \,, \\
\label{eqn:liminfIK}\liminf_{\epsilon \rightarrow 0}  \bigg(\frac{x_1(\epsilon )-x_1(0)-  \epsilon x_2(0)}{\sqrt{\epsilon^3 \log \log \epsilon^{-1}}} \bigg) &= -M.  
\end{align}
Note that~\eqref{eqn:limsupIK} implies that for any $\delta >0$ there exists a (random) sequence  $\epsilon_n = \epsilon_n(\omega) > 0$  such that 
$\epsilon_n \to 0$ as $n \to \infty$ and 
\begin{align*}
x_1(\epsilon_n) \geq x_1(0)+ \epsilon_n x_2(0) + \sqrt{\epsilon_n^3 \log \log \epsilon_n^{-1}} (M-\delta)  \,.
\end{align*}
 Similarly, using~\eqref{eqn:liminfIK}; for any $\delta >0$ there exists $\bar{\epsilon}_n = \bar{\epsilon}_n(\omega)>0$ such that 
 $\bar{\epsilon}_n \to 0$ as $n \to \infty$ and 
\begin{align*}
x_1(\bar{\epsilon}_n) \leq x_1(0)+ \bar{\epsilon}_n x_2(0) + \sqrt{\bar{\epsilon}_n^3 \log \log \bar{\epsilon}_n^{-1}} (-M+\delta)  \,.
\end{align*}

Now suppose $x_2(0) > 0$ as the case $x_2(0) < 0$ is treated similarly. 
 Since $\sqrt{\epsilon^3 \log \log \epsilon^{-1}} \ll \epsilon$ for small $\epsilon$, we obtain that 
 $x_1(t) \geq x_1(0)$ for all small times $t$. Intuitively, if we rewrite our system as \eqref{eqn:dtk}, then since $x_0 = 0$ one has $x_2 < c = x_2(0)$ for all small times. Thus, $x_1$ is increasing 
 for small times as our analysis shows. Observe that if there was a noise in the $x_1$ coordinate, then it would change sign on the time scale $\sqrt{\epsilon} \gg \epsilon$, and therefore 
 $x_1(t) - x_1(0)$ would change sign as well.  
 \end{Example}

\begin{Example}
As our next example, we consider the following diffusion on $\RR^2$
\begin{equation}\label{eqn:sde2d}
\begin{aligned}
dx_1 &= (x_1^2 -x_2^2)\, dt \,, \\
 dx_2 &= 2 x_1 x_2 \, dt + dB_t \,,
\end{aligned}
\end{equation}
where $x_t=(x_1(t), x_2(t))$ has initial condition $x_0=(x_1(0), x_2(0))=(0,0)$ and $B_t$ is a standard, one-dimensional Brownian motion on $(\Omega, \mathcal{F},\PP)$.  This particular diffusion has been extensively studied (see~\cite{AKM_12, BHW_12, GHW_11, HerMat_15i}).  In particular, one of the main results in these works is that the diffusion defined by~\eqref{eqn:sde2d} is non-explosive for all initial conditions in $\RR^2$.  This is true despite the fact that the associated deterministic dynamics (obtained by deleting $dB_t$ from~\eqref{eqn:sde2d}) explodes in finite time when started from $(s, 0)$ with $s>0$.  Here, 
we  study the behavior at time zero of the first coordinate, $x_1$ (the second one is trivial).   

First, define the family $\{ \Phi_\alpha \}_{\alpha \succ0}$ of weak contractions centered at $(0,0)\in \RR^2$ by 
\begin{align*}
\Phi_{\alpha}(y_1,y_2) = \big( y_1 \alpha_1^{-1}, y_2 \alpha_2^{-1}\big).  
\end{align*}
We also define, for $\epsilon_*>0$ small enough, the asymptotic index $\psi: [0, \epsilon_*]\rightarrow [0, \infty)^2$ as 
\begin{align*}
\psi(\epsilon) = (\epsilon^2 \log \log \epsilon^{-1}, \sqrt{\epsilon \log \log \epsilon^{-1}}) \,.
\end{align*}
Then, by \eqref{eqn:beps1} and \eqref{eqn:sigeps1} the process 
\begin{align*}
y_t^\epsilon = \Phi_{\psi(\epsilon)}(x_{\epsilon t})
\end{align*}
satisfies the SDE
\begin{align*}
dy_1&= \epsilon^3 \log \log \epsilon^{-1} y_1^2 \, dt - y_2^2 \, dt,\\
dy_2&= 2\epsilon^3 \log \log \epsilon^{-1} y_1y_2\, dt + \frac{1}{\sqrt{r(\epsilon)} }\, dB_\epsilon \,,
\end{align*}
where $(y_1(0), y_2(0))=(0,0)$ and $B_\epsilon$ is a standard, one-dimensional Brownian motion.  Note that 
\begin{align*}
b_\epsilon (y_1, y_2) &= ( \epsilon^3 \log \log \epsilon^{-1} y_1^2  - y_2^2, \, 2\epsilon^3 \log \log \epsilon^{-1} y_1 y_2 )\end{align*}  
and $\sigma_\epsilon \equiv \sigma$ is a constant matrix.  Since
\begin{align*}
b_\epsilon(y_1, y_2) \rightarrow b(y_1, y_2)= (-y_2^2, 0)
\end{align*}  
uniformly on compact subsets of $\RR^2$ as $\epsilon \rightarrow 0$, it follows that Assumption~\ref{assump:1} is satisfied
and by Lemma \ref{lem:1}, Assumption \ref{assump:2} is also satisfied.  Thus,  Theorem~\ref{thm:main} and Corollary~\ref{cor:1} apply for any sufficiently small $t_*\in (0,1]$.  However, because the limiting ODE
\begin{align*}
\dot{y}_1 &= - y_2^2  \,,\\
\dot{y}_2&= \dot{f} \,,
\end{align*}
is well-defined for all times $t\in [0, 1]$ for any $f\in \C^0([0, 1]; \RR^2)$ (simply integrate it), we may take $t_*=1$.  

With a slight abuse of notation, we let $\pi :\RR^2\rightarrow \RR$ denote the projection onto the first coordinate, and $F: \EE([0, 1]; \RR^2)\rightarrow \RR \cup \{\Delta\}$ be such that $F (g)= \pi g_{1}$ if $g_{1} \in \RR^2$ and $\Delta$ otherwise (cf. Example \ref{ex:IK}).  Corollary~\ref{cor:1} implies that $\{ F(y^\epsilon ) \}_{\epsilon \in (0, \epsilon_*]}$ is relatively compact in $\RR \cup \{ \Delta \}$.  Moreover, if $J_2: \C_0([0, 1]; \RR)\rightarrow (-\infty, 0]$ is given by 
\begin{align*}
J_2(f) = -\int_0^{1} f_s^2 \, ds,
\end{align*}        
then the a.s. limit set of $F(y^\epsilon)$ as $\epsilon \rightarrow 0$ is given by
\begin{align*}
F(\K_0) = \{ J_2(f) \, : \,  \textstyle{\frac{1}{2}\int_0^1 (\dot{f}_s)^2 \, ds} \leq 1\}.  
\end{align*}
Note that
\begin{align*}
-M:= \inf\{ J_2(f)\, : \, \textstyle{\frac{1}{2}\int_0^1 (\dot{f}_s)^2 \, ds} \leq 1\} \in (-\infty, 0),  
\end{align*}
since $H^1([0,1]) \hookrightarrow L^\infty([0,1])$, and consequently almost surely
\begin{align*}
\liminf_{\epsilon \rightarrow 0} \frac{x_1(\epsilon )}{\epsilon^3 \log \log \epsilon^{-1}} = -M \,, \qquad 
\limsup_{\epsilon \rightarrow 0} \frac{x_1(\epsilon )}{\epsilon^3 \log \log \epsilon^{-1}} =0.  \end{align*}

\end{Example}  

 \begin{Example}
Next, we consider the following Lorenz 96 model with $d=5$
\begin{equation}\label{eqn:LOR}
\begin{aligned}
dx_1&= (x_2 - x_4) x_5 \, dt -x_1 \, dt + dB_1 \,,\\
 dx_2&= (x_3-x_5) x_1 \, dt -x_2 \, dt + dB_2 \,, \\
 dx_3&= (x_4-x_1) x_2 \, dt -x_3 \, dt \,, \\
 dx_4 &= (x_5-x_2) x_3 \, dt -x_4 \, dt  \,, \\
 dx_5 &= (x_1-x_3) x_4 \, dt -x_5 \, dt \,,
\end{aligned}
\end{equation}
where the process $x_t=(x_1(t), x_2(t), \ldots, x_5(t))$ above is assumed to evolve on $\RR^5$ starting initially at $x_0=0$, and $B_i$, $i=1,2$, are independent, standard Brownian motions defined on $(\Omega, \mathcal{F},\PP)$.  Using nearly identical computations, one can also treat the general Lorenz 96 model in $d$ dimensions, but for clarity we prefer the concrete scenario above.  
Our interest  stems from the fact that the nonlinearities mimic pairwise interactions in the Navier-Stokes equation.  Here we will analyze the small-time behavior of $x_5$.   
 
 In this example, again the family $\{ \Phi_\alpha\}_{\alpha \succ 0}$ of weak contractions centered at $0\in \RR^5$ is the same  as above
 \begin{align*}
 \Phi_\alpha(y) = (y_1 \alpha_1^{-1}, y_2 \alpha_2^{-1}, \ldots, y_5 \alpha_5^{-1} ).  
 \end{align*}
For $\epsilon_*>0$ small enough, we define $\psi:[0, \epsilon]\rightarrow [0, \infty)^5$ by 
 \begin{align*}
\psi(\epsilon)= ( \sqrt{\epsilon \log \log \epsilon^{-1}}, \sqrt{\epsilon \log \log \epsilon^{-1}}, \epsilon^2 \log \log \epsilon^{-1}, \sqrt{\epsilon^7 (\log \log \epsilon)^3}, \epsilon^{5} (\log \log \epsilon^{-1})^2) 
\end{align*}
and note that $\psi$ is an asymptotic index for~\eqref{eqn:LOR}.

Let $y^\epsilon$ be given by $y^\epsilon_t = \Phi_{\psi(\epsilon)}(x_{\epsilon t})$ and note that $y^\epsilon$ solves the following SDE
\begin{align*}
dy_1&= b_{1, \epsilon}(y_1, y_2, y_4, y_5) \, dt + \frac{1}{\sqrt{r(\epsilon)}} \, dB_{1, \epsilon} \,,\\
dy_2&=  b_{2, \epsilon}(y_1, y_2, y_3, y_5) \, + \frac{1}{\sqrt{r(\epsilon)}} \, dB_{2, \epsilon} \,, \\
dy_3&=-y_1 y_2 \, dt + b_{3, \epsilon}(y_2, y_3, y_4) \, dt  \,,\\
dy_4&= -y_2 y_3 \, dt + b_{4, \epsilon}(y_3, y_4, y_5) \, dt \,,\\
dy_5&=y_1 y_4 \, dt + b_{5, \epsilon}(y_3, y_4, y_5) \, dt\,, 
\end{align*}   
where $b_{i,\epsilon}\rightarrow 0$, $i=1,2,3,4,5$, as $\epsilon \rightarrow 0$ uniformly on compact subsets in $\RR^5$.  Furthermore, one can check that the $b_{i, \epsilon}$ are locally Lipschitz on $\RR^5$, so that Assumption~\ref{assump:1} is satisfied.  Thus Theorem~\ref{thm:main} and Corollary~\ref{cor:1} both apply.  In this case, the limiting ODE is 
\begin{gather} \label{eqn:rel}
\dot{y}_1= \dot{f}_1 \,,\qquad 
\dot{y}_2=  \dot{f}_2 \,,\qquad 
\dot{y}_3 =-y_1 y_2 \,,\qquad 
\dot{y}_4= -y_2 y_3 \,,\qquad 
\dot{y}_5= y_1 y_4 \,,
\end{gather}   
with $y(0) = 0$.

Note that we can solve \eqref{eqn:rel} explicitly and take $t_*=1$ again.  
Indeed, let $\pi_5 :\RR^5\rightarrow \RR$ denote the projection onto the fifth coordinate and, slightly abusing notation again, define 
$F: \EE([0,1]; \RR^5)\rightarrow \RR \cup \{ \Delta \}$ by $F(g) = \pi_5 g_1$ if $g_1 \in \RR^5$ and $F(g)=\Delta$ otherwise.  Applying Corollary~\ref{cor:1}, we note that $\{ F(y^\epsilon) \}_{0< \epsilon \leq \epsilon_*}$ is relatively compact in $\RR\cup \{ \Delta\}$.  Furthermore, define
$J_3: \C_0([0,1]; \RR^2)\rightarrow \RR$ as $y_5$ in \eqref{eqn:rel}
\begin{align*}
J_3(f_1, f_2) = \int_0^1 f_1(t) \left( \int_0^t f_2(s) \left( \int_0^s f_1(r) f_2(r) \, dr\right) \, ds\right) dt \,,
\end{align*}  
where $f_1(0) = f_2(0) = 0$.  Let
\begin{align*}
M&= \sup \{ J_3(f_1, f_2) \,: \, \textstyle{\frac{1}{2}\int_0^1 |(\dot{f}_1, \dot{f}_2)|^2\leq 1} \}\\
m&= \inf\{ J_3(f_1, f_2) \, : \, \textstyle{\frac{1}{2}\int_0^1 |(\dot{f}_1, \dot{f}_2)|^2\leq 1}\} \,.
\end{align*}
Setting $\dot{f} = \dot{g} = 1$ we see that $J_3$ clearly attains positive values.  However, seeing that $J_3$ can realize negative is not immediately obvious. Nevertheless, by choosing $f_1(t) = \sin(5t)$ and $f_2(t) = \sin(t)$
one has $J_3(f_1, f_2) \approx -0.00605$. Thus by a proper rescaling to guarantee the constraint on $\dot{f}_1$ and $\dot{f}_2$, we obtain 
then (almost surely)
\begin{align*}
\limsup_{\epsilon \rightarrow 0 } \frac{x_5(\epsilon)}{\epsilon^{5} (\log \log \epsilon^{-1})^2}&= M \,,\\
\liminf_{\epsilon \rightarrow 0 } \frac{x_5(\epsilon)}{\epsilon^{5} (\log \log \epsilon^{-1})^2}&= m \,,
\end{align*}
for some $M > 0$ and $m < 0$. 
\end{Example}

\begin{Remark}
Because the LIL at time zero is local phenomena, there is nothing important about the SDEs above being defined on all of $\RR^d$.  One can consider the same equations (or different ones) defined in a neighborhood $U$ of the initial condition $x$, with scalings taking the same forms.  One then applies the general results of this paper, but with $\ustar$ and $U_\epsilon$ in place of all of $\RR^d$.  See also Example~\ref{ex:1}.       
\end{Remark}

\begin{Remark} 
 Let $W_t$ be a standard, real-valued Brownian motion defined on $(\Omega, \mathcal{F}, \PP)$. Chung's LIL~\cite{Chung_48} for $W_t$ at time zero states that, $\PP$-almost surely, 
 \begin{equation}\label{e.ChungsLIL.small}
\liminf_{\varepsilon \rightarrow 0}\bigg\{ \sqrt{\frac{\log \log  \varepsilon^{-1}}{\varepsilon}} \max_{0 \leqslant t \leqslant \varepsilon } \vert W_t \vert \bigg\}= \frac{\pi}{\sqrt{8}}.  
\end{equation}
A natural question is whether Theorem~\ref{thm:main} and Corollary~\ref{cor:1} can be used to deduce a similar results for SDEs in the examples above? The short answer is that the results do apply, but do not obviously capture the precise asymptotic behavior that one would expect for the $\liminf$ of the running maximum.  To see why in more detail, we have provided the next example.  
\end{Remark}  

\begin{Example}
Consider~\eqref{eqn:itkol} for $d=2$ and define the maximum process for the first coordinate $x_1^*(t):=\sup_{s\in [0,t]} |x_1(s)|$.  Using the scaling property~\cite[equation~(2.1)]{KC_98} 
$$
x_1^*(t) \overset{d}{=} t^{3/2} x_1^*(1) \qquad \forall t\geq 0
$$ along with inversion,  one of the main results in~\cite{KC_98}
states that, $\PP$-almost surely,  
\begin{align}
\label{eqn:davar}
\liminf_{\epsilon\rightarrow 0}\,  \frac{x_1^*(\epsilon)}{\phi(\epsilon)} =c, \qquad \phi(\epsilon):= \frac{\epsilon^{3/2}}{(\log \log \epsilon^{-1})^{3/2}},
\end{align}
for some deterministic constant $c\in (0,\infty)$.  To investigate properties of $x^*$ in our framework, let $\pi:\RR^2\rightarrow \RR$ be the projection onto the first coordinate, and let 
$F:\EE([0,1]; \RR^2)\rightarrow \RR\cup \{\Delta\}$ be a continuous mapping given by $F(g)= \sup_{s\in [0,1]} | \pi g_s|$ if $g_1\in \RR^2$ and $F(g)=\Delta$ if $g_1=\Delta$.  Following Example~\ref{ex:IK}, 
if $J:\C^0([0, 1]; \RR)\rightarrow \RR$ is given by
\begin{align}
J(f)= \sup_{0\leq t\leq 1}\bigg| \int_0^t f_s \, ds \bigg|  \,,
\end{align}   
then the almost sure limit set of $F(y^\epsilon)$ as $\epsilon \rightarrow 0$ is 
\begin{align*}
F(\K_0) = \{ J(f) \, : \, \textstyle{\frac{1}{2}\int_0^1 (\dot{f}_s)^2 \, ds} \leq 1 \}.
\end{align*}
Here, we recall that $y^\epsilon$ is as in~\eqref{def:yepsik} with $d=2$.  From this we deduce that, almost surely,
\begin{align}
\label{eqn:chung1}
&\limsup_{\epsilon \rightarrow 0 } \frac{x_1^*(\epsilon)}{\epsilon^{3/2} \sqrt{\log\log \epsilon^{-1}}} = M\\
\label{eqn:chung2}&\liminf_{\epsilon \rightarrow 0 } \frac{x_1^*(\epsilon)}{\epsilon^{3/2} \sqrt{\log\log \epsilon^{-1}}} =0,
\end{align}
for some constant $M\in (0, \infty)$.  Note that while relation~\eqref{eqn:chung1} provides precise asymptotics, the second relation~\eqref{eqn:chung2}, though consistent with~\eqref{eqn:davar}, does not establish~\eqref{eqn:davar}.  This is because the $\liminf$ of $x^*$ acts on a smaller scale which is not captured by our scaling.

\end{Example}

 \section{Large Deviations}\label{sec:ld}

In this section, we outline two key results that are used to prove Theorem~\ref{thm:main}.  Both results follow almost immediately from the existing literature.  Here, we only provide slight adjustments, if needed, 
to connect with our setting. 

The first result stated below is a consequence of~\cite[Theorem~1.1]{Bal_88}, which   is an improvement of~\cite[2.4 Th\'{e}or\`{e}me, Chapitre III]{Az_80}.  
The only difference here is that we are not assuming ~\cite[(A.3)]{Bal_88}, but our proof follows   
 nearly identical localization procedure for an open set $\ustar$ as opposed to $\RR^d$ in \cite{Bal_88}.  
 
 Recall that $(x, f) \mapsto S_x(f)$ is the solution operator of \eqref{eqn:det}, $r(\epsilon) = \log\log \epsilon^{-1}$, 
 and the distance-like function $d_t$ is defined in \eqref{eqn:dfmd}.  We also recall that for every $\epsilon >0$, $B_t^\epsilon$ denotes a standard Brownian motion on $\RR^k$,
  and that $\epsilon_*>0$ and $\ustar$ were fixed in Assumption \ref{assump:1}.

  \begin{Theorem}
  \label{thm:LD1}
  Suppose that Assumption~\ref{assump:1} is satisfied and let $K\subset \ustar$ be a compact set  and $t\in (0,1]$.  For every $\rho >0, R>0, a>0$, there exist 
  $\epsilon_0  \in (0, \epsilon_*], \alpha >0$ such that for all $x\in K, f\in \C^0([0, t]; \RR^k)$ and $g=S_x(f)$ with 
  \begin{align*}
  \int_0^t |\dot{f}_s|^2 \, ds \leq a \qquad \text{ and } \qquad g([0, t]) \subset K,
  \end{align*}   
  we have the following estimate
  \begin{align*}
  \PP \bigg\{d_{t}\Big(\tfrac{1}{\sqrt{r(\epsilon)}} B^\epsilon, f\Big) \leq \alpha\,\, \text{ and }\,\, d_{t}(y^\epsilon, g) > \rho \bigg\} \leq e^{-Rr(\epsilon)}
  \end{align*} 
  for all $0<\epsilon \leq \epsilon_0$.  
  \end{Theorem}

  \begin{proof}
  Since $K$ is compact, there is $\rho_0\in (0, \rho]$ and a compact set $K' \subset \ustar$ containing a $2 \rho_0$ neighborhood of $K$.    
  Let $V\subset \RR^d$ be a bounded, open set such that $K' \subset V \subset \ustar$ and let $\varphi:\RR^d\rightarrow [0,1]$ be $C^\infty$ function with $\varphi=1$ on $K'$ and $0$ on $V^c$.  
  For any function $p$ on $\ustar$ we define $p_\varphi$ on $\RR^d$ by  
\begin{align*}
 p_\varphi(y) = \begin{cases}
  p(y) \varphi(y) & \qquad y \in \ustar \,,\\
  0 & \qquad y \notin \ustar \,.
  \end{cases}
  \end{align*}    
 Thus, to $b, \sigma, b_\epsilon, \sigma_\epsilon$ we associate respectively $b_\varphi, \sigma_\varphi, b_{\varphi,\epsilon}, \sigma_{\varphi,\epsilon}$.  Then, by construction and Assumption~\ref{assump:1},  $b_\varphi, \sigma_\varphi, b_{\varphi,\epsilon}, \sigma_{\varphi,\epsilon}$ are bounded and globally Lipschitz on $\RR^d$ and  $b_{\varphi,\epsilon} \rightarrow b_\varphi$, $\sigma_{\varphi,\epsilon}\rightarrow \sigma_\varphi$ uniformly on $\RR^d$ as $\epsilon \rightarrow 0$.  Let $y^\epsilon_\varphi$ be the unique solution of the It\^{o} SDE
  \begin{align*}
  \begin{cases}
  dy_t = b_{\varphi,\epsilon}(y_t) \, dt + \frac{1}{\sqrt{r(\epsilon)}} \sigma_{\varphi,\epsilon}(y_t) \,dB^{\epsilon}_t ,\\
 y_0=x.  
 \end{cases}
 \end{align*}
By standard arguments, the solution above belongs almost surely to  $\C_x([0, t]; U)$.  
  Then~\cite[Theorem 1.1]{Bal_88} with $h(\epsilon)\equiv 1$ and $\epsilon^{-2}$ replaced by $r(\epsilon) <  \epsilon^{-2}$,
  provides the existence of $\epsilon_0>0, \alpha >0$ such that for all $x\in K$, $f\in \C^0([0,t]; \RR^k)$ and $g=S_x(f)$ with 
  \begin{align*}
  \int_0^t |\dot{f}_s|^2 \, ds \leq a \qquad \text{ and } \qquad g([0, t]) \subset K
  \end{align*}   
  we have, for every $\epsilon \in (0, \epsilon_0]$
  \begin{align*}
  \PP \bigg\{ d_{ t}\Big(\tfrac{1}{\sqrt{r(\epsilon)} }B^\epsilon, f\Big) \leq \alpha\,\, \text{ and }\,\, d_{t}(y^\epsilon_\varphi, g) > \rho_0 \bigg\} \leq e^{- Rr(\epsilon)} .
  \end{align*}
 However, $y^\epsilon_\varphi$ and $y^\epsilon$ coincide until the first time both $y^\epsilon_\varphi$ and $y^\epsilon$ exit $K'$.  Since $g([0, t]) \subset K$ and $\rho > \rho_0$
 \begin{align*}
  \PP \bigg\{ d_{ t}\Big(\tfrac{1}{\sqrt{r(\epsilon)} }B^\epsilon, f\Big) \leq \alpha\,\, \text{ and }\,\, d_{t}(y^\epsilon, g) > \rho \bigg\}&\leq 
    \PP \bigg\{ d_{ t}\Big(\tfrac{1}{\sqrt{r(\epsilon)} }B^\epsilon, f\Big) \leq \alpha\,\, \text{ and }\,\, d_{t}(y^\epsilon, g) > \rho_0 \bigg\} \\
    &\leq 
  \PP \bigg\{ d_{ t}\Big(\tfrac{1}{\sqrt{r(\epsilon)} }B^\epsilon, f\Big) \leq \alpha\,\, \text{ and }\,\, d_{t}(y^\epsilon_\varphi, g) > \rho_0 \bigg\} \\
  &\leq e^{- Rr(\epsilon)} 
  \end{align*}
and the proof is finished.  
  \end{proof}
  
  Recall the Cramer transform $\lambda:\EE([0, 1]; \ustar)\rightarrow [0, +\infty]$ introduced in~\eqref{def:lambda}, and for any $A\subset \EE_x([0,1]; \ustar)$ Borel set define 
  \begin{align}
  \Lambda(A)= \inf_{g\in A} \lambda(g).  
  \end{align} 

  \begin{Theorem}
  \label{thm:LD2}
  Suppose that Assumption~\ref{assump:1} is satisfied.  For any Borel set $A\subset \EE_x([0,1];\ustar)$, 
  \begin{align}
  -\Lambda(\text{\emph{interior}}(A)) \leq \liminf_{\epsilon \rightarrow 0}\frac{1}{ r(\epsilon)} \log \PP \{ y^\epsilon \in A\} \leq \limsup_{\epsilon \rightarrow 0} \frac{1}{r(\epsilon)} \log \PP \{ y^\epsilon \in A \} \leq - \Lambda(\overline{A})\,,  
  \end{align}
 where $\text{\emph{interior}}(A)$ and $\overline{A}$ respectively denote the interior and closure of $A$.  
  \end{Theorem}
  
  For the proof of Theorem \ref{thm:LD2}, we refer to \cite[proof of 2.13 Th\'{e}or\`{e}me, Chapitre III]{Az_80}
with $\epsilon^{-2}$ replaced by $r(\epsilon)$, which works in our setting as one merely needs $\sigma$ to be locally Lipschitz on $\ustar$ 
  rather than $C^1$.

 Given the previous two results, we are now prepared to prove Theorem~\ref{thm:main}.

  \section{Proof of Theorem~\ref{thm:main}} 
  \label{sec:proof}

The proof of Theorem~\ref{thm:main} is similar to the proof of the main result in Baldi~\cite{Bal_86} and also Caramellino~\cite{Car_98}, but with a different topology and and set of assumptions.  

We first need some auxiliary results.  

\begin{Proposition}
\label{prop:1}
Suppose that Assumption~\ref{assump:1} is satisfied and $x\in \ustar$, $t_*$, and $\K_x(t_*)$ are as in the statement of Theorem~\ref{thm:main}.  Then the following assertions hold. 
\begin{itemize}
\item[(i)]  For any $c\in (0,1)$:
\begin{align*}
\PP\bigg\{\lim_{j \rightarrow \infty} d_{t_*}(y^{c^j}, \K_x(t_*) ) = 0\bigg\}=1. 
\end{align*}
\item[(ii)]  If Assumption~\ref{assump:2} (i) is furthermore satisfied, then 
\begin{align*}
\PP\bigg\{\lim_{\epsilon\rightarrow 0} d_{t_*}(y^{\epsilon}, \K_x(t_*) ) = 0\bigg\}=1. 
\end{align*}\end{itemize}
\end{Proposition}

\begin{proof}
To show (i),  fix $\delta >0$ and consider the set 
\begin{align}
\K_{x,\delta} = \{ g \in \EE_x([0,1]; \ustar) \, : \, d_{t_*}(g, \K_x(t_*)) \geq \delta\}.  
\end{align}
First we claim that $\K_{x,\delta} $ is closed in $\EE([0,1]; \ustar)$.  Indeed, if $g_n\in \K_{x, \delta}$ converges $g_n\rightarrow g \in \EE([0,1]; \ustar)$ as $n\rightarrow \infty$ and $g(t_*)=\Delta$, then $d_{t_*}(g, \K_x(t_*))=\infty$ 
since, by the definition of $t^*$, $h(t^*) \neq \Delta$ for any $h \in \K_x(t_*)$. 
Hence, $g\in \K_{x,\delta}$.  
If on the other hand $g(t_*)\in \ustar$, then by the continuity of $g$, $\tau_1(g) > t_*$ and by the definition of convergence in $\EE([0,1]; \ustar)$ one has that 
$\tau_1(g_n) > t_*$ for any sufficiently large $n$.  Then, the triangle inequality \eqref{eqn:trin} and the definition of 
 $d_{t_*}$ imply for any $h\in \K_{x}(t_*)$
\begin{equation}
d_{t_*}(g_n, g)+ d_{t_*}(g, h) \geq d_{t_*}(g_n, h)\geq \delta.  
\end{equation}
Passing $n\rightarrow \infty$,  we obtain $g \in \K_{x,\delta}$.

We now claim that there exists $\delta'>0$ for which $\Lambda(\K_{x, \delta}) > 1+ \delta'$.  Suppose to the contrary that 
$\Lambda(\K_{x, \delta}) \leq 1$.  By definition, there exists a sequence $g_n \in  \K_{x, \delta}$ such that 
\begin{align*}
\lim_{n\rightarrow \infty} \lambda(g_n) = \Lambda(\K_{x,\delta})\leq 1.  
\end{align*}
Thus, for all $n$ large enough, $g_n\in \mathcal{M}:=\{ g\in \EE_x([0,1]; \ustar)\,: \, \lambda(g) \leq 2\}$.  The set $\mathcal{M}$ is sequentially compact since   
$\mathcal{M}$ is the image of the sequentially compact set (the compact Sobolev embedding $H^1 \hookrightarrow \mathcal{C}$) 
\begin{align*}
\C_2= \{ f \in \C([0,1]; \RR^k) \, : \,\, \tfrac{1}{2}\textstyle{\int_0^1} |\dot{f}_s|^2 \, ds \leq 2\}
\end{align*}   
under the continuous mapping $S_x: \C_2 \rightarrow \EE_x([0,1]; \ustar)$ given by $S_x(f)$.   Hence, the sequence $\{g_n \}$ has a convergence subsequence $\{ g_{n_k} \}$ converging to some 
$g \in \K_{x, \delta}$.  
The lower semicontinuity of $\lambda$ then implies
\begin{align*}
1\geq \liminf_{k\rightarrow \infty} \lambda( g_{n_k}) \geq \lambda(g).  
\end{align*}
In particular,  $g\in \K_x$, contradicting closedness and the definition of $\K_{x, \delta}$.  Thus we have shown that there is $\delta'>0$ so that $\Lambda(\K_{x, \delta}) > 1+ \delta'$.  

By Theorem~\ref{thm:LD2} and the fact that $\K_{x, \delta}$ is closed
\begin{align*}
\limsup_{\epsilon\rightarrow 0} \frac{1}{r(\epsilon)} \log \PP \{ y^\epsilon \in \K_{x, \delta} \} \leq -(1+\delta'). 
\end{align*} 
Hence, using the definition of $r(\epsilon)$ we have for all $j$ large and $c\in (0,1)$
\begin{align}\label{eqn:lbq}
\PP \{ y^{c^j} \in \K_{x, \delta} \} \leq \frac{C}{j^{1+\tfrac{\delta'}{2}}}
\end{align}  
for some constant $C = C(c) >0$.  The Borel-Cantelli lemma then implies   
\begin{align*}
\PP \bigg\{ \limsup_{j\rightarrow \infty}\,  d_{t_*}(y^{c^j}, \K_x(t_*)) \geq \delta \bigg\} =0.  
\end{align*}
Since $\delta >0$ was arbitrary, $\lim_{j\to \infty} d_{t_*}(y^{c^j}, \K_x(t_*)) =0$ almost surely, finishing the proof of part (i).  

To establish part (ii), fix $\delta >0$ and by Assumption~\ref{assump:2}(i) choose a constant $c\in (0,1)$ and a (random) index $J_0=J_0(\omega, c)>0$ such that for all $j\geq J_0$ and $\epsilon \in [c^{j+1}, c^j]$ 
\begin{align}\label{eqn:lbp}
y_{t_*}^\epsilon \in \ustar\qquad \text{ and } \qquad 
d_{t_*}(y^{c^j}, y^\epsilon) < \frac{\delta}{2}.
\end{align} 
By \eqref{eqn:lbq}, we can increase $J_0$ if necessary so that $j\geq J_0$ implies 
\begin{align*}
d_{t_*}(y^{c^j}, \K_x(t_*) ) < \frac{\delta}{2}  . 
\end{align*}
Using \eqref{eqn:lbp}, the triangle inequality \eqref{eqn:trin}, for any $j\geq J_0$ and $\epsilon \in [c^{j+1}, c^j]$ one has
\begin{align*}
d_{t_*}(y^\epsilon, \K_x(t_*))< \delta\,,
\end{align*}
and part (ii) follows.  
  \end{proof}

\begin{Proposition}
\label{prop:2}
Suppose that Assumption~\ref{assump:1} is satisfied.  Let $g\in \K_x$ be such that $\lambda(g) <1$.  Then, for all $\epsilon >0$ and $c\in (0,1)$ we have 
\begin{align*}
\PP\{ d_{t_*}(y^{c^j}, g) < \epsilon \,\text{ for infinitely many }j \} =1.  
\end{align*}
\end{Proposition}

\begin{proof} 
In the proof, we abbreviate \emph{for infinitely many} $j$ as \emph{i.o.} $j$.   
Fix $g\in \K_x$ with $a:= \lambda(g)<1$ and fix $\epsilon >0$ and $c\in (0,1)$.  Since the infimum in the definition of $\lambda$ (see \eqref{def:lambda})
is attained, there exists $f\in \C_a([0, 1]; \RR^k)$ so that $g=S_x(f)$.  By shifting $f$ by a constant value, we may assume without loss of generality that $f_0=0$, as the time derivative is invariant under this shift.  For $a_*>0$ and $c\in (0,1)$ define events 
\begin{align*}
F_j = \left\{ d_{t_*}\left(\tfrac{1}{\sqrt{r(c^j)}} B^{c^j}, f \right) < a_* \right\} \qquad \text{ and }\qquad  H_j= \{ d_{t_*}(y^{c^j} , g ) < \epsilon\}.  
\end{align*}  
Then, Theorem~\ref{thm:LD1} implies that there exists $a_*>0$ and $J>0$ such that $j\geq J$ implies 
\begin{align*}
\PP\{ F_j \cap H_j^c\} \leq \exp( -2 r(c^j))\leq \frac{C}{j^2}
\end{align*}
for some constant $C>0$.  The Borel-Cantelli lemma then implies  $\PP\{ F_j \cap H_j^c \,\text{ i.o.} j\}=0$.  Now, by Mueller~\cite{Mue_81} or Gantert~\cite{Gan_93} we have $1= \PP\{ F_j \, \text{ i.o. } j \}$.  Thus, 
\begin{align*}
1= \PP\{ F_j \, \text{ i.o. }j \} \leq \PP \{ F_j \cap H_j \,\text{ i.o. }j \} + \PP \{ F_j \cap H_j^c \, \text{ i.o. }j \} \leq \PP\{ H_j \, \text{ i.o. }j \} \,,
\end{align*}
as desired. 
\end{proof}

We will also need the following topological result.  
\begin{Lemma}
\label{lem:comp}
Let $t\in (0,1]$, $V\subset \RR^d$ be open, $x\in V$ and suppose $K\subset \C_x([0, t]; V)$.  Then, $K$ is compact in $\C_x([0, t]; V)$ if and only if $K$ is compact in $\EE_x([0, t]; V)$.  
\end{Lemma}

\begin{proof}
Let $K\subset \C_x([0, t]; V)$ be compact in $\EE_x([0, t]; V)$. To show that $K$ is compact in $\C_x([0, t]; V)$, 
it suffices to prove that if $\mathcal{V}$ is open in $\C_x([0, t]; V)$, then $\mathcal{V}$ is open in $\EE_x([0, t]; V)$.   
Equivalently, we  show that $\mathcal{W}=\EE_x([0, t]; V)\setminus \mathcal{V}$ is closed in $\EE_x([0, t]; V)$.  Let $g_n \in \mathcal{W}$ be such that $g_n\rightarrow g\in \EE_x([0,t]; V)$.  If $g(t) = \Delta$, then clearly $g\notin \mathcal{V} \subset \C_x([0, t]; V)$, and therefore $g\in \mathcal{W}$.  On the other hand, if $g(t) \in V$, then $g\in \C_x([0, t]; V)$ by definition of $\EE_x([0,t]; V)$. 
 By definition of the topology on $\EE_x([0, t]; V)$, $g_n \in \mathcal{W} \cap \C_x([0, t]; V)$ for all $n\geq N$, $N>0$ large enough.  Since $\mathcal{W}\cap \C_x([0, t]; V)$ is closed
 and $g_n \to g$ in the topology of  $\C_x([0, t]; V)$, then $g\in \mathcal{W}$.  
 Thus,  $\mathcal{W}$ is closed in $\EE_x([0, t]; V)$.

Conversely, suppose $K\subset \C_x([0, t]; V)$ is compact in $\C_x([0,t]; V)$. To prove that $K$ is compact in $\EE_x([0, t]; V)$, it is enough to show that if $\mathcal{V}$ is open in $\EE_x([0, t]; V)$, then 
$\mathcal{V} \cap \C_x([0, t]; V)$ is open in $\C_x([0, t]; V)$.  Note that this follows immediately from the fact that if $g_n \to g$ in   $\C_x([0, t]; V)$, then 
$g_n \to g$ in $\EE_x([0, t]; V)$. 
 \end{proof}         

Given the previous three results, we next prove the main general result, Theorem~\ref{thm:main}.

\begin{proof}[Proof of Theorem~\ref{thm:main}]
We have already established part (ii) in Proposition~\ref{prop:1}(ii).  To prove part (i), we need to show that $\Y(\omega)$ is relatively compact, almost surely.      
By Proposition~\ref{prop:1}(ii), for any $\delta >0$ there exists $\epsilon_0=\epsilon_0(\omega, \delta) \in (0, \epsilon_*]$ such that $d_{t_*}(y^\epsilon, \K_x(t_*)) \leq \delta$ for all $\epsilon \in (0,\epsilon_0]$.  
In particular, for any small $\delta > 0$, $y^\epsilon( t_*) \neq \Delta$ a.s. for any $\epsilon \in [0, \epsilon_0]$. Hence, $y^\epsilon \in \C_x([0,t_*]; \ustar)$ for any  $\epsilon \in (0, \epsilon_0]$,
and therefore  the closure of $\{ y^\epsilon\}_{0< \epsilon \leq \epsilon_0}$ in $\EE_x([0, t_*]; \ustar)$ is the same as the  closure in $\C_x([0,t_*]; \ustar)$.

We claim that the closure of $\{ y^\epsilon\}_{0< \epsilon \leq \epsilon_0}$ in $\EE_x([0, t_*]; \ustar)$ in $\C_x([0,t_*]; \ustar)$  is compact in  
$\C_x([0, t_*]; \ustar)$, almost surely.  Consequently,  the closure of $\{ y^\epsilon\}_{0< \epsilon \leq \epsilon_0}$ is compact in $\EE_x([0, t_*]; \ustar)$ almost surely by Lemma~\ref{lem:comp} finishing the proof of part (i) of the result. 
To prove the claim, fix a sequence $\{\epsilon_n\} \subset (0, \epsilon_0]$. By passing to a subsequence, we can suppose that $\epsilon_n \to \epsilon_\infty \in [0, \epsilon_0]$. If $\epsilon_\infty > 0$, then by Assumption \ref{assump:2}(ii), $y^{\epsilon_n} \to y^{\epsilon_\infty}$
in  $\EE_x([0, t_*]; \ustar)$, almost surely.  Note that this convergence happens in $\C_x([0, t_*]; \ustar)$ almost surely by the definition of $\epsilon_0$.  If $\epsilon_\infty = 0$, then 
by Proposition~\ref{prop:1}(ii), $y^{\epsilon_n} \to \K_x(t_*)$ almost surely in $\C_x([0, t_*]; \ustar)$. However, since $\K_x(t_*)$ is compact, there exists a subsequence of   $\{\epsilon_n\}$,  again denoted by 
$\{\epsilon_n\}$, such that $y^{\epsilon_n}$ converges almost surely to $g \in \K_x(t_*)$, and therefore $g$ belongs to the closure of $\{ y^\epsilon\}_{0< \epsilon < \epsilon_0}$. Our claim now follows.  

Next, for part (iii), Proposition~\ref{prop:1} implies that the a.s. limit set $\K'$ of $\Y$ as $\epsilon \rightarrow 0$ is contained in $\K_x(t_*)$. On the other hand, Proposition~\ref{prop:2}
yields that any point of $g \in \K_x(t_*)$ with $\lambda(g) < 1$ belongs to $\K'$ almost surely.  We therefore need to check that any $g\in \K_x(t_*)$ with $\lambda(g)=1$ belongs to $\K'$ as well.  Note that for every $g \in \K_x(t_*)$ the infimum in \eqref{def:lambda} is attained.  Thus there exists $f\in \C^0([0,1]; \RR^k)$ with 
\begin{align*}
\tfrac{1}{2}\int_0^1 |\dot{f}_t|^2 dt \leq 1 \,\,\, \text{ and } \,\,\,S_x(f) = g. 
\end{align*}
For each $n > 1$ define $f_n = (1 - 1/n)f$ and let $g_n = S_x(f_n)$. Clearly $\tfrac{1}{2}\int_0^1 |(\dot{f}_n)_t|^2 dt < 1$, and therefore $g_n \in \K'$
for each $n$.  By \eqref{eqn:cnt}, $S_x$ is continuous on the unit ball in $H^1$, and therefore $g \in \K'$ since limit sets are closed.    
\end{proof}

\section{Proof of Lemma~\ref{lem:1} and Corollary~\ref{cor:2}}
\label{sec:lem}

In this section, we prove Lemma~\ref{lem:1} and Corollary~\ref{cor:2}, which give basic criteria for the family $\{y^\epsilon\}_{0<\epsilon\leq \epsilon_*}$ to satisfy Assumption~\ref{assump:2}.  We begin with the:     

\begin{proof}[Proof of Lemma~\ref{lem:1}]
We first prove Assumption~\ref{assump:2}(i).  Recall the definition of $t_*=t_*(x)\in (0, 1]$ and $L \subset \ustar$ in~\eqref{eqn:tdef} 
and recall that $\K_x(t_*)$ is compact in both $\C_x([0, t_*]; \ustar)$ and $\EE_x([0,t_*]; \ustar)$ by Lemma~\ref{lem:comp}.  From~\eqref{eqn:tdef},  it  follows that
$g([0, t_*]) \subset L \subset \ustar$ for all $g\in \K_x(t_*)$.
Choose a compact set $L'\subset \ustar$ with $L\subset \inter (L')$ and fix $\delta >0$ such that $2\delta < \textrm{dist}(L', \partial \ustar)$. 
Using compactness of $\K_x(t_*)$ in $\C_x([0,t_*]; \ustar)$, the Arzel\` a-Ascoli theorem
implies that the set $\K_x(t_*)$ is equicontinuous, and therefore for any  
$c\in (0,1)$ sufficiently close to $1$ one has  
\begin{align}\label{eqn:eqi}
 \sup_{h\in \K_x(t_*)}    \sup_{\substack{t\in [0,t_*]\\s\in [c t, t]}}   | h(t) - h(s) |< \frac{\delta}{3}.
 \end{align}
By properties (i) and (ii) in the definition of an asymptotic index (Definition \ref{def:2}), we can decrease $\epsilon_0$ if necessary such that $|\psi(u)| < \kappa$ for any $u \in [0, \epsilon_0]$, where 
$\kappa$ is as in Definition \ref{def:1}(iii).  
Then, since $\{\Phi_\alpha\}$ is a family of weak contractions centered at $x$,  for all 
$c\in (0,1)$ close enough to $1$, there exists $J_1>0$ (deterministic) such that for all $j\geq J_1$ and all $\epsilon \in [c^{j+1}, c^j]$ one has $\ustar \subset U_{\psi(c^j)}$ and 
 \begin{align}\label{eqn:sem}
  | \Phi_{\psi(\epsilon)} \circ \Phi_{\psi(c^j)}^{-1}(y) - y| +2 | \Phi_{\psi(c^{j+1})}\circ \Phi_{\psi(c^j)}^{-1} (y)-y | < \frac{\delta}{3} 
  \end{align}
  for all $y\in L'$. Fix 
$c\in (0, 1)$ sufficiently close to 1 such that \eqref{eqn:eqi} and \eqref{eqn:sem} are satisfied.

By Proposition~\ref{prop:1}(i), for almost every $\omega$ we can choose a finite $J_0=J_0(\omega, c) \geq J_1$ such that $j\geq J_0$ implies $y^{c^j}_{s}\in L'$ for any $s \in [0, t_*]$.  
Also, from \eqref{eqn:trin} for all $j\geq J_0$ and $\epsilon \in [c^{j+1}, c^j]$ it follows that
\begin{align}\label{eqn:iin}
d_{t_*}(y^\epsilon, y^{c^j })&\leq  d_{t_*} \big(\Phi_{\psi(\epsilon)} \circ \Phi_{\psi(c^j)}^{-1}(y^{c^j}) , y^{c^j}) + d_{t_*} (\Phi_{\psi(\epsilon)} (x_{\epsilon t}),  \Phi_{\psi(\epsilon)} (x_{c^j t})).
\end{align} 
Let us prove that the second term on the right hand side is finite.  First, observe that  $$\Phi_{\psi(\epsilon)} (x_{c^j t}) =  \Phi_{\psi(\epsilon)} \circ \Phi_{\psi(c^j)}^{-1}(y^{c^j}_t).$$  Thus, since $y^{c^j}_{s}\in L'$ for any $s \in [0, t_*]$, one has by~\eqref{eqn:sem} 
\begin{align}
\label{eqn:rem}
d_{t_*}(\Phi_{\psi(\epsilon)} \circ \Phi_{\psi(c^j)}^{-1}(y^{c^j}_\cdot), y^{c^j}_\cdot)< \frac{\delta}{3}.  
\end{align}
In particular, since $2\delta < \text{dist}(L', \partial U^*)$, $\Phi_{\psi(\epsilon)} (x_{c^j t}) \in \ustar$ for any $t \in [0, t^*]$.  

Next, since $\psi$ is an asymptotic index and $\{ \Phi_\alpha\}_{\alpha >0}$ is a family of weak contractions centered at $x$, one has for any $t \leq t^*$, $t < \tau_{t_*}(y^\epsilon)$
\begin{align*}
 |\Phi_{\psi(\epsilon)} (x_{\epsilon t}) - \Phi_{\psi(\epsilon)} (x_{c^j t}) |& \leq  |\Phi_{\psi(c^{j+1})} (x_{\epsilon t}) - \Phi_{\psi(c^{j+1})} (x_{c^j t}) | \\
& \leq   \sup_{s\in [c t, t]}      |\Phi_{\psi(c^{j+1})} (x_{c^j s}) -\Phi_{\psi(c^{j+1})} (x_{c^j t}) |\\
&\leq  \sup_{s\in [c t, t]}    \bigg\{  |\Phi_{\psi(c^{j+1})} (x_{c^j s}) - y^{c^j}_s| + |y^{c^j}_t -\Phi_{\psi(c^{j+1})} (x_{c^j t}) |  + | y^{c^j}_t-y^{c^j}_s |\bigg\} \\
& \leq 2  d_{t_*}\big(\Phi_{\psi(c^{j+1})}\circ \Phi_{\psi(c^j)}^{-1} (y^{c^j}), y^{c^j} )  + \sup_{\substack{t\in [0,t_*]\\s\in [c t, t]}}  | y^{c^j}_t-y^{c^j}_s |. \end{align*}
Consequently, from \eqref{eqn:eqi} and \eqref{eqn:sem} 
for any $g\in \K_x(t_*)$ and $t\leq t_*$ it follows
\begin{align*}
|\Phi_{\psi(\epsilon)} (x_{\epsilon t}) - \Phi_{\psi(\epsilon)} (x_{c^j t}) |& \leq 2 d_{t_*} \big( \Phi_{\psi(c^{j+1})}\circ \Phi_{\psi(c_j)}^{-1} (y^{c^j}), y^{c^j} \big) + 2 d_{t_*} (y^{c^j},   g)\\
& \qquad + \sup_{h\in \K_x(t_*)}    \sup_{\substack{t\in [0,t_*]\\s\in [c t, t]}}   | h(t) - h(s) | \\
&\leq \frac{2\delta}{3} + 2 d_{t_*} (y^{c^j},   g).  
\end{align*}
Since the left hand side is independent of $g$, we can take the infimum with respect to $g \in \K_x(t_*)$ and obtain 
\begin{equation}
|\Phi_{\psi(\epsilon)} (x_{\epsilon t}) - \Phi_{\psi(\epsilon)} (x_{c^j t}) | \leq  \frac{2\delta}{3} +  2 d_{t_*} (y^{c^j} , \K_x(t_*)). 
\end{equation}  
Thus,   
  by Proposition~\ref{prop:1}, by increasing $J_2$ if needed, for any $j\geq J_2$ and $t\in [0, t_*]$ we have 
\begin{equation}\label{eqn:fbb}
|\Phi_{\psi(\epsilon)} (x_{\epsilon t}) - \Phi_{\psi(\epsilon)} (x_{c^j t}) | \leq \delta \,.
\end{equation}  
Consequently, since $2\delta < \text{dist}(L', \partial U^*)$,  employing~\eqref{eqn:rem} we obtain that $y^\epsilon_t \in U^*$ for all $t\in [0, t_*]$.   

Returning to \eqref{eqn:iin} and using \eqref{eqn:sem}, definition of $d_{t_*}$, and \eqref{eqn:fbb}, we obtain for any $j \geq J_2$ 
\begin{align}
d_{t_*}(y^\epsilon, y^{c^j })&\leq  d_{t_*} \big(\Phi_{\psi(\epsilon)} \circ \Phi_{\psi(c^j)}^{-1}(y^{c^j}) , y^{c^j}) + \sup_{t \in[0, t_*]} |\Phi_{\psi(\epsilon)} (x_{\epsilon t}) - \Phi_{\psi(\epsilon)} (x_{c^j t})| \leq \frac{4\delta}{3} 
\end{align}
and  Assumption~\ref{assump:2}(i) follows.

To prove Assumpion~\ref{assump:2}(ii), fix $\epsilon_\infty \in (0, \epsilon_*)$, let $\omega$ be a realization of the noise, and $t_1 \in [0, \tau_{t^*}(y^{\epsilon_\infty}(\omega)))$ with $t_1 \leq t_*$. 
The assertion follows once we prove that for any sequence $\epsilon_n$ with $\epsilon_n \to \epsilon_\infty$ as $n\rightarrow \infty$ one has  $y^{\epsilon_n}(\omega) \to y^{\epsilon_\infty}(\omega)$ as $n \to \infty$ in the space $\C_x([0, t_1]; \ustar)$ for almost all $\omega$. To simplify the notation, 
we will often drop the argument $\omega$ below. 
Since $t_1 < \tau_{t^*}(y^{\epsilon_\infty})$, then $\delta_0 := \inf_{t \in [0, t_1]} \textrm{dist}(y^{\epsilon_\infty}_t, \partial \ustar) > 0$. In particular $y^{\epsilon_\infty} ([0, t_1]) \subset K$ for some compact $K=K(\omega)$ with $K \subset \ustar$ almost surely. 
Fix $\delta \in (0, \delta_0]$ and  large enough $n$ such that $\epsilon_n > \epsilon_\infty/2$ and 
\begin{equation}\label{eqn:smrb}
|\Phi_{\psi(\epsilon_n)}\circ \Phi^{-1}_{\psi(\epsilon_\infty)}(y) - y| < \frac{\delta}{3}
\end{equation} 
for any $y \in K$. Note that such $n$ exists since $\epsilon_n \to \epsilon_\infty$, $\{\Phi_\alpha\}_{\alpha \succ 0}$ is a sequence of weak contractions and $\psi$ an asymptotic index. 

Then, from \eqref{eqn:smrb} and  the fact that $y\mapsto \Phi_{\psi(\epsilon_\infty/2)}(y) \in C^2(U)$, we obtain for any $t \leq t_1$ with $t <  \tau_{t^*}(y^{\epsilon_n})$
\begin{equation}\label{eqn:idnd}
\begin{aligned}
| y^{\epsilon_n}_t - y^{\epsilon_\infty}_t| &\leq |\Phi_{\psi(\epsilon_n)}(x_{\epsilon_n t}) - \Phi_{\psi(\epsilon_n)}(x_{\epsilon_\infty t})| + |\Phi_{\psi(\epsilon_n)}\circ \Phi^{-1}_{\psi(\epsilon_\infty)}(y^{\epsilon_\infty}_t) - y_t^{\epsilon_\infty}| \\
&\leq  |\Phi_{\psi(\epsilon_\infty/2)}(x_{\epsilon_n t}) - \Phi_{\psi(\epsilon_\infty/2)}(x_{\epsilon_\infty t})| + \frac{\delta}{3}.  
\end{aligned}
\end{equation}
Using the fact that 
\begin{equation}\label{eqn:xcnt}
s\mapsto x_s:[0, t]\rightarrow \EE_x([0, t]; \ustar)
\end{equation}
is continuous, $\PP$-almost surely, we obtain that for any sufficiently large $n$,   $\textrm{dist}(y^{\epsilon_n}_t, \partial \ustar) \geq \frac{\delta_0}{2}$ for any $t \leq t_1$,  $t <  \tau_{t^*}(y^{\epsilon_n})$.  
Again, a standard extension argument implies that $\textrm{dist}(y^{\epsilon_n}_t, \partial \ustar) \geq \frac{\delta_0}{2}$ for any $t \leq t_1$, and in particular $t_1< \tau_{t^*}(y^{\epsilon_n})$. Finally, passing $n \to \infty$
using \eqref{eqn:xcnt}, and since $\delta > 0$ was arbitrary, the assertion follows. 
\end{proof}

We now turn our attention to the proof of Corollary~\ref{cor:2}.    

\begin{proof}[Proof of Corollary~\ref{cor:2}]
Fix $c\in (0,1)$, $\delta >0$, and $\epsilon \in [c^{j+1}, c^j]$.  By Proposition~\ref{prop:1}(i) 
there is an almost surely finite random variable $J_0=J_0(c, \omega)$ such that 
\begin{align*}
y^{c^j}([0, t_*]) \subset \text{interior}(L) \,\,\text{ for all } \,\, j \geq J_0 \,,
\end{align*}
where $t_*$ and $L$ are as in \eqref{eqn:tdef} (see also Definition \ref{def:2}). 
Define $T_{\epsilon} = T_{\epsilon}(\omega) = \inf\{ t\geq 0 \, : \, y^\epsilon_t \notin \text{interior}(L) \} $ and for every $t \geq  0$ set $T_\epsilon(t)= t\wedge T_\epsilon$.  
Then, for any $j\geq J_0$ and $t\in [0,t_*]$ we obtain 
\begin{align}
\nonumber \sup_{s \leq T_\epsilon(t)} | y^\epsilon_s - y^{c^j}_s |  &\leq \int_0^{T_\epsilon(t)} | b_\epsilon(y^\epsilon_s) - b_{c^j}(y^{c^j}_s)| \, ds + \sup_{s\leq t_*} \bigg| \frac{\sigma_\epsilon}{ \sqrt{\epsilon r(\epsilon)}} \, B_{\epsilon s} - \frac{\sigma_{c^j}}{\sqrt{c^j r(c^j)}} \, B_{c^js} \bigg|\\
&=:S_1 + S_2(j, t_*).   \label{eqn:ses} 
\end{align}    
To estimate $S_1$, note that since both $y^\epsilon, y^{c^j}$ map $[0, T_\epsilon(t_*)]$ to 
the compact set $L \subset \ustar$ for $j\geq J_0$, we have by Assumption~\ref{assump:1}
\begin{align*}
S_1 &\leq  \int_0^{T_\epsilon(t)} |b_\epsilon(y^\epsilon_s)- b(y^\epsilon_s) |\, ds + \int_0^{T_\epsilon(t)}  |b_{c_j}(y^{c^j}_s)- b(y^{c^j}_s) | \, ds+ \int_0^{T_\epsilon(t)} |b(y^\epsilon_s) - b(y^{c^j}_s)| \, ds\\
& \leq 2 C_j(L) t_*  + C_L \int_0^{t} \sup_{v\leq T_\epsilon(s)} |y^\epsilon_v - y^{c^j}_v| \, ds  
\end{align*}
 where $t\leq t_*$, $C_j(L), C_L>0$ are deterministic constants and $C_j(L)\rightarrow 0$ as $j\rightarrow \infty$.  Combining the previous estimate with \eqref{eqn:ses} and using Gronwall's inequality gives 
 \begin{align*}
 \sup_{s \leq T_\epsilon(t_*)} | y^\epsilon_s - y^{c^j}_s |  &\leq ( 2 C_j(L) t_* + S_2(j, t_*)) e^{C_L t_*} 
  \end{align*} 
  for any $j\geq J_0$.

 In order to estimate $S_2(j, t_*)$, for any multindex $\alpha >0$ belonging to $\RR^k$, let $\Phi_\alpha:\RR^k\rightarrow \RR^k$ and $\psi:[0, \epsilon_*]\rightarrow \RR^k$ be given by 
 \begin{align*}
 \Phi_\alpha(y)= (y_1\alpha_1^{-1}, \ldots, y_k \alpha_k^{-1}) \qquad \text{and} \qquad \psi(\epsilon) =(\sqrt{\epsilon \log\log\epsilon^{-1}}, \ldots, \sqrt{\epsilon \log\log\epsilon^{-1}}) \,,
 \end{align*}     
 where $\epsilon_* \in (0, e^{-1})$.  By Example \ref{ex:1}, $\{ \Phi_\alpha\}_{\alpha \succ 0}$ is a family of weak contractions centered at $0$ in $\RR^k$ while by 
 Example \ref{ex:2}, for $\epsilon_*>0$ sufficiently small, $\psi:[0, \epsilon_*]\rightarrow [0, \infty)^k$ is an asymptotic index.  We can then estimate $S_2(j, t_*)$ as follows  
 \begin{align*}
 S_2(j, t_*) & \leq \sup_{t\leq t_*} \bigg| \sigma_\epsilon \big( \Phi_{\psi(\epsilon)}(B_{\epsilon t} )- \Phi_{\psi(c^j)}(B_{c^j t})\big) \bigg| + \sup_{t \leq t_*} |(\sigma_\epsilon - \sigma_{c^j}) \Phi_{\psi(c^j)} (B_{tc^j})|\\
 &\leq \|\sigma_\epsilon \| d_{t_*}( \Phi_{\psi(\epsilon)}(B_{\epsilon \cdot }), \Phi_{\psi(c^j)}(B_{c^j \cdot}))
+  D_{j} \sup_{t\leq t_*} \frac{|B_{tc^j}|}{\sqrt{c^j \log \log c^{-j}}} \,,
 \end{align*} 
 where $D_{j}$ is a  deterministic constant with $D_{j}\rightarrow 0$ as $j\rightarrow \infty$ and $\| \cdot \|$ denotes the matrix norm.    
Now, the assumptions of Lemma \ref{lem:1} are satisfied with $x_t = B_t$ solving \eqref{eqn:sdemain} with $\tilde{b} = 0$ and $\tilde{\sigma}$ being the $d\times d$ identity matrix.   
 For any $M>0$, by Lemma~\ref{lem:1} there is $c\in (0, 1)$ and $J_1=J_1(\omega,c, M)>0$ such that for any $j\geq J_1$ and $\epsilon \in [c^{j+1}, c^j]$ we obtain
  \begin{align*}
  d_{t_*}( \Phi_{\psi(\epsilon)}(B_{\epsilon \cdot }), \Phi_{\psi(c^j)}(B_{c^j \cdot})) \leq \frac{1}{M (\|\sigma\|+1)}.  
  \end{align*}
By the standard LIL for Brownian motion, for any $c\in (0,1)$, there exists $J_2=J_2(\omega, c)>0$ such that for any $j\geq J_2$ it follows that, almost surely,
 \begin{align*}
 \sup_{t\leq t_*} \frac{|B_{tc^j}|}{\sqrt{c^j \log \log c^{-j}}} \leq 2.  
 \end{align*}
Overall, for any $M>0$, there exists $\delta' > 0$ such that for any $c\in (1-\delta', 1)$ there is $J_3=J_3(\omega, c, M)$ such that for all $j\geq J_3$ one has
\begin{align*}
 \sup_{s \leq T_\epsilon(t_*)} | y^\epsilon_s - y^{c^j}_s |  &\leq ( 2 C_j(L) t_* + M^{-1} \|\sigma^\epsilon\|D_{j}) e^{C_L t_*}  < \delta 
\end{align*}
for all $\epsilon \in [c^{j+1}, c^j]$.  
By increasing $M$ and $J_3$ if necessary, Proposition \ref{prop:1} part (i) ensures $T_\epsilon > t_*$, and therefore $T_\epsilon (t_*) = t_*$, so that Assumption~\ref{assump:2}(i) is satisfied.    

In order to establish Assumption~\ref{assump:2} part (ii), fix $\epsilon_0 \in (0, \epsilon_*]$ and $\omega$ in a subset of $\Omega$ of full measure specified below and let 
$t < \tau_{t_*}(y_\cdot^{\epsilon_0}(\omega))$
or $t = t^*$ if $y_{t_*}^{\epsilon_0}(\omega) \neq \Delta$. 
Since $s\mapsto y_s^{\epsilon_0}(\omega) : [0,t]\rightarrow \ustar$ is continuous, there is a compact set $K\subset \ustar$ such that 
the image of $[0, t]$ under the map $s \mapsto y_s^{\epsilon_0}(\omega)$
is contained in $K$.  

Fix any $\epsilon \in (0, \epsilon_*]$ and set $S_\epsilon(\omega):=\inf\{ t\geq 0 \, : \, y^\epsilon_t \notin K \}$ and $S_\epsilon(\omega, s)=S_\epsilon(\omega) \wedge s$ for any $s > 0$.  
To simplify the notation, we drop the explicit dependence on $\omega$ and proceed as above to find that for any $s\leq t \leq  t_*$
\begin{align*}
\sup_{v\leq  S_\epsilon(s)} | y_v^\epsilon - y_v^{\epsilon_0}|  &\leq \int_0^{S_\epsilon(s)} |b_{\epsilon_0}(y_v^{\epsilon_0}) - b_{\epsilon_0}(y_v^\epsilon)| \, dv + 
\int_0^{S_\epsilon(s)} |b_{\epsilon_0}(y_v^\epsilon) - b_\epsilon(y_v^\epsilon)| \, dv\\
&\qquad + \sup_{v\leq t} \left| \frac{\sigma_{\epsilon_0}}{\sqrt{{\epsilon_0} r({\epsilon_0})}} B_{v{\epsilon_0}} - \frac{\sigma_\epsilon}{\sqrt{\epsilon r(\epsilon)}} B_{v\epsilon}\right| \\
& \leq C_{\epsilon, {\epsilon_0}}(K) t_* 
+ C_{\epsilon_0}(K) \int_0^{S_\epsilon(s)} \sup_{w\leq  S_\epsilon(v)} | y_w^\epsilon - y_w^{\epsilon_0}| \, dv\\
&\qquad +
 \sup_{v\leq t} \left| \frac{\sigma_{\epsilon_0}}{\sqrt{{\epsilon_0} r({\epsilon_0})}} B_{v{\epsilon_0}} - \frac{\sigma_\epsilon}{\sqrt{\epsilon r(\epsilon)}} B_{v\epsilon}\right|   \,,
 \end{align*}     
where $C_{\epsilon, {\epsilon_0}}(K)$, $C_{\epsilon_0}(K)$ are deterministic constants with $C_{\epsilon, \epsilon_0}(K)\rightarrow 0$ as $\epsilon \rightarrow \epsilon_0$.  Gronwall's inequality, $t_* \leq1$, and 
and the almost sure path continuity of Brownian motion
then imply
\begin{align*}
\sup_{v\leq  S_\epsilon(t)} | y_v^\epsilon - y_v^{\epsilon_0}|  &\leq C'_{\epsilon, {\epsilon_0}}(K)e^{C_{\epsilon_0}(K) t} 
\end{align*} 
for some $C'_{\epsilon, \epsilon_0}$ such that $C'_{\epsilon, \epsilon_0}\rightarrow 0$ as $\epsilon \rightarrow \epsilon_0$.  
By continuity, and $t < \tau_{t_*}(y_\cdot^{\epsilon_0}(\omega))$ we obtain for $\epsilon$ sufficiently close to $\epsilon_0$ that $S_{\epsilon}(t) = t$, and therefore
 \begin{align*}
\sup_{v\leq  t} | y_v^\epsilon - y_v^{\epsilon_0}|  &\leq C_{\epsilon, {\epsilon_0}}(K)e^{C_{\epsilon_0}(K) t} \,.
\end{align*} 
for any $\epsilon$ sufficiently close to $\epsilon_0$. Passing $\epsilon \to \epsilon_0$ and using the definition of convergence in $\EE_x([0,t_*]; \ustar)$, we obtain the desired result. 
\end{proof}

\section{Application: Criteria for Regular Points on boundary of a bounded domain in~$\RR^d$}\label{sec:regular}

Throughout this section, for simplicity we suppose $U=U^*=\RR^d$ and $V\subset \RR^d$ is a non-empty, open set.  We moreover suppose that  $\partial V:= \overline{V}\setminus V\subset \RR^d$
 is non-empty, where $\overline{V}$ denotes the closure of $V$ in $\RR^d$.  

For $x\in \partial V$, our goal is to use Theorem~\ref{thm:main} to deduce criteria for the diffusion $x_t$ solving~\eqref{eqn:sdemain} to be 
\emph{regular} at $x$.  Specifically, we say that $x\in \partial V$ is \emph{regular for $(x_t, V)$} if 
$$
\PP_x\{ \tau_{\overline{V}} >0 \} =0\,,
$$ where 
\begin{align}
\tau_{\overline{V}}= \inf\{ t> 0 \, : \, x_t \notin \overline{V} \} \,.
\end{align}       
We call $x$ \emph{irregular for} $(x_t, V)$ otherwise.    

\begin{Remark}
Note that $x\in \partial V$ irregular for $(x_t, V)$ means that $x_t$ spends, with positive probability, a positive  amount of time in $\overline{V}$ before exiting $\overline{V}$.  Because the event $\{ \tau_{\overline{V}} >0 \}$ belongs to the germ $\sigma$-field $\bigcap_{t>0} \mathcal{F}_t$, Blumenthal's $0$-$1$ law implies that this event either has probability $0$ or $1$.  Thus $x\in \partial V$ is irregular for $(x_t, V)$ if and only if $\PP_x\{ \tau_{\overline{V}} >0 \} =1$.       
\end{Remark}

In order to state the main result of this section we recall (cf. \eqref{eqn:det}) the deterministic system associated to~\eqref{eqn:sdeas}:
\begin{equation}\label{eqn:detagain}
\left\{
\begin{aligned}
\dot{g}_t &= b(g_t) + \sigma(g_t) \dot{f}_t \,, \\
g_0 &= x.
\end{aligned}
\right. 
\end{equation}
In this section, we view~\eqref{eqn:detagain} as a control problem, with controls $f$ belonging to the class (cf. \eqref{eqn:dca})
\begin{align}\label{eqn:dcoa}
\C_1= \{ f\in \C^0([0,1]; \RR^k) \, : \, \textstyle{\frac{1}{2}\int_0^1 |\dot{f}_s|^2 \,ds} \leq 1\}.  
\end{align}
Let $L$ be a compact set containing a neighborhood of $x\in \partial V$ and let $t_*>0$ be as in~\eqref{eqn:tdef}.  For any $t\in [0,t_*]$ and any $f\in \C_1$, let $S_x^t(f)\in L$ denote the solution of~\eqref{eqn:detagain} at time $t$.  For $t\in [0,t_*]$, let  
\begin{align}
\A(x,t) = \{ y\in U \, : \, S_x^t(f)=y \text{ for some }  f\in \C_1 \} \,,
\end{align}  
where $\C_1$ is as in~\eqref{eqn:dcoa}
and 
\begin{align}
\A(x, \leq t_*) = \bigcup_{t\in [0, t_*]} \A(x, t).  
\end{align}

Consider the processes $x_t$ and $y_t^\epsilon$ defined by~\eqref{eqn:sdemain} and~\eqref{eqn:sdeas}, respectively, both having initial condition $x\in \partial V$.  

\begin{Definition}
We say that a point $z\in \overline{V}^c$ is \emph{asymptotically invariant at} $x$ if there exists $\delta >0$ such that the following statement holds almost surely:
whenever  $y_t^\epsilon \in B_\delta(z)$ ($y^\epsilon$ solves \eqref{eqn:sdeas}) for some $t\in (0,t_*]$ and some $\epsilon \in (0, \epsilon_*]$, then $x_{\epsilon t} \in \overline{V}^c$. 
 \end{Definition}  
 
  Let $\mathcal{I}_x\subset \overline{V}^c$ be the set of asymptotically invariant points at $x$.

\begin{Example}
Let $x=0\in \RR^d$ and 
\begin{align*}
V=\{ y\in \RR^d \, : \, y_d < 0 \}.
\end{align*}
 Suppose that $\Phi_\alpha :\RR^d \rightarrow \RR^d$, $\alpha >0$, is of the form
\begin{align*}
\Phi_\alpha(y)=(y_1 \alpha_1^{-1}, y_2 \alpha_2^{-1}, \ldots, y_d \alpha_d^{-1})
\end{align*} 
and $\psi:[0, \epsilon_*]\rightarrow [0, \infty)^d$ is an asymptotic index.  If $y_t^\epsilon:= \Phi_{\psi(\epsilon)}(x_{\epsilon t})$,  
then any $z\in\overline{V}^c$ is asymptotically invariant at $x$.  Indeed, choose any $\delta >0$ such that $B_\delta(z) \subset V^c$ and use that $\psi_d(\epsilon) > 0$.   
\end{Example}

We have the following result, which is a consequence of Theorem~\ref{thm:main}. 
\begin{Theorem}
\label{thm:regular}
Suppose that Assumption~\ref{assump:1} and Assumption~\ref{assump:2} are satisfied.  Then, $x\in \partial V$ is regular for $(x_t, V)$ if 
\begin{align}
\A(x, \leq t_*) \cap \mathcal{I}_x \neq \emptyset. 
\end{align}

\end{Theorem}
\begin{proof}
Suppose $z\in \A (x, \leq t_*) \cap \mathcal{I}_x$.  Since $z$ is asymptotically invariant, there 
exists $\delta >0$ such that whenever $y_t^\epsilon \in B_\delta(z)$ we have $x_{\epsilon t} \in \overline{V}^c$, almost surely. 
Since $z \in  \A (x, \leq t_*)$, there exists $t\in (0, t_*]$ and an $f\in \C_1$ such that $S_x^t(f)=z\in \overline{V}^c.$  Let $g=S_x^\cdot(f) \in \C^0([0, t_*]; \RR^d)$ and note that by definition,  $g \in \K_x(t_*)$.  
 If $\lambda(g) = 1$, then as in the proof of Theorem \ref{thm:main}, we can find $g^*$ with $\lambda(g^*) < 1$ and $\sup_{t \in [0, t^*]}|g_t - g_t^*| < \delta/2$. If, on the other hand, $\lambda(g) < 1$ we 
 simply set $g^* = g$. 
 By Proposition~\ref{prop:2}, there exists a deterministic sequence $\epsilon_n> \epsilon_{n+1}>0$ with $\epsilon_n \rightarrow 0$ such that 
\begin{align*}
\PP\left\{ d_{t_*}( y^{\epsilon_n}, g^*) < \frac{\delta}{2} \text{ for infinitely many } n \right\} =1,  
\end{align*}  
and consequently 
\begin{align*}
\PP\left\{ d_{t_*}( y^{\epsilon_n}, g) < \delta \text{ for infinitely many } n \right\} =1 \,.  
\end{align*}  
Since $g_t = z$, then 
\begin{align*}
\PP\left\{ |y^{\epsilon_n}- z| < \delta \text{ for infinitely many } n \right\} =1. 
\end{align*}  
Thus by the asymptotic invariance of $z$, $x_{\epsilon_n t} \in \overline{V}^c$ for infinitely many $n$, almost surely.  Hence,     
\begin{align*}
\PP_x\{ \tau_{\overline{V}} = 0 \} \geq \PP\{ d_{t_*}( y^{\epsilon_n}, g) < \delta \text{ for infinitely many } n \}=1  \end{align*} 
and the proof is finished.       
\end{proof}

The next two results provide sufficient conditions on the noise that guarantee a given boundary point is regular.   Before proceeding, we let $\mathcal{R}(A)$ denote the range of the matrix $A$, or equivalently the space spanned by the columns of $A$. 

\begin{Proposition}
\label{prop:invariance}
Let $x\in \partial V$ and suppose that $\tilde{b}\in C^\infty(\RR^d; \RR^d)$ and $\tilde{\sigma}\in C^\infty(\RR^d; M_{d\times k})$.  
Assume there exists $v \in \mathcal{R}(\tilde{\sigma}(x))$ satisfying the following two properties:
\begin{itemize}
\item[(qi)] $x+ \lambda v\in \overline{V}^c$ for all $\lambda \in (0,1]$.  
\item[(qii)]  For all $\lambda \in (0,1]$, there exists $\delta_\lambda >0$ such that if $\delta_{\epsilon, \lambda}= \sqrt{\epsilon \log \log \epsilon^{-1}} \delta_\lambda$ and $\lambda_\epsilon = \lambda\sqrt{\epsilon \log \log \epsilon^{-1}}$, then $$B_{\delta_{\epsilon, \lambda}}(x+ \lambda_\epsilon v) \subset \overline{V}^c$$ for all $\epsilon >0$ small enough.    
\end{itemize}
Then, $x$ is regular for $(x_t, V)$.    
\end{Proposition}

\begin{proof}
For any multiindex $\alpha >0$, let $\Phi_\alpha^1, \Phi_\alpha: \RR^d\rightarrow \RR^d$ be given by 
\begin{align*}
\Phi_\alpha^1(y) = (y_1 \alpha_1^{-1}, \ldots, y_d \alpha_d^{-1}) \qquad \text{ and } \qquad \Phi_\alpha(y) = \Phi_\alpha^1(y-x)+ x. 
\end{align*}
By Example~\ref{ex:1},  $\{ \Phi_\alpha\}_{\alpha >0}$ is a family of weak contractions centered at $x\in \partial V \subset \RR^d$.  
If we define, for $\epsilon_*>0$ small enough, $\psi:[0, \epsilon_*]\rightarrow [0, \infty)^d$ as
\begin{align*}
\psi(\epsilon) = (\sqrt{\epsilon \log \log \epsilon^{-1}}, \ldots, \sqrt{\epsilon \log \log \epsilon^{-1}}) \,, 
\end{align*}
then by Example~\ref{ex:2}, $\psi$ is an asymptotic index.  For $\epsilon \in (0, \epsilon_*]$, define 
\begin{align}
y_t^\epsilon = \Phi_{\alpha(\epsilon)}(x_{\epsilon t}) \,.
\end{align}
We thus see that for all $t< \tau_1(x_\cdot) \epsilon^{-1}$ (see \eqref{eqn:bup} for the definition of $\tau_1(x_\cdot)$), $y_t^\epsilon$ satisfies an SDE of the form~\eqref{eqn:sdeas} (c.f. 
\eqref{eqn:coc}--\eqref{eqn:sigeps1}) and it is easy to check that 
\begin{align}
b_\epsilon(y) \rightarrow 0 \qquad \text{ and } \qquad  \sigma_\epsilon(y) \rightarrow \tilde{\sigma}(x) 
\end{align}
as $\epsilon \rightarrow 0$ for every $y\in \RR^d$ with the convergence above uniform on compact subsets of $\RR^d$.  
Furthermore, $b_\epsilon, \sigma_\epsilon$ are locally Lipschitz on $\RR^d$ for every $\epsilon \in (0, \epsilon_*]$.  
Thus Assumption~\ref{assump:1} and, by Lemma~\ref{lem:1}, Assumption~\ref{assump:2} are both satisfied.  

The associated deterministic system is
\begin{equation}\label{eqn:ald}
\left\{
\begin{aligned}
\dot{g}_t &= \tilde{\sigma}(x) \dot{f}_t  \,, \\
g_0 &=x \,,  
\end{aligned}
\right.
\end{equation} 
where $f \in \C_1 = \{ h \in \C^0([0,1]; \RR^k) \, : \, \tfrac{1}{2}\textstyle{\int_0^1} |\dot{f}_s|^2 \, ds \leq 1 \}$. 
Note that \eqref{eqn:ald} has constant coefficients as $x\in \partial V$ is the initial condition, which is fixed. 
 Let $v\in \mathcal{R}(\tilde{\sigma}(x))$  satisfy (qi) and (qii).  In particular, $v= \tilde{\sigma}(x) w$ for some $w\in \RR^k$.  Hence, for any small enough
  $\lambda \in (0, 1]$,  $f_t := \lambda t w \in \C_1$ and then $g_t = x + \lambda v t$.  Also, there is $\lambda^0 > 0$ such that  all points $z\in \RR^d$ of the form
\begin{align*}
z= x+  \lambda v \qquad \lambda \in (0, \lambda^0]
\end{align*}              
belong to $\mathcal{A}(x, \leq 1)$.  Since $\overline{V}^c$ is open, there exists $\delta_{\lambda^0} >0$ such that $B_{\delta_{\lambda^0}}(x+\lambda^0 v)\subset \overline{V}^c$, and then by 
 property (qii), $B_{\delta_{\epsilon,\lambda^0}}(x+\lambda_{\epsilon}^0 v)\subset \overline{V}^c$ 
for all $\epsilon >0$ small enough.  But, almost surely, $y_t^\epsilon\in B_{\delta_{\lambda^0}}(x+\lambda^0 v)$  if and only if $x_{\epsilon t} \in  B_{\delta_{\epsilon,\lambda^0}}(x+\lambda_{\epsilon}^0 v) \subset \overline{V}^c$.  Hence, 
\begin{align*}
x + \lambda^0 v \in \mathcal{A}(x, \leq 1) \cap \mathcal{I}_x \neq \emptyset
\end{align*} 
and the proof follows from 
Theorem~\ref{thm:regular}.          
\end{proof}

\begin{Corollary}
\label{cor:noise}
Let $x\in \partial V$ and suppose that $\tilde{b}\in C^\infty(\RR^d; \RR^d)$ and $\tilde{\sigma}\in C^\infty(\RR^d; M_{d\times k})$.  Assume that there exists a unit vector $n(x)$ and 
$\delta >0$ such that $B_{\delta}(x+ \delta n(x))$ is tangent to $\partial V$ at $x$ and is contained in $\overline{V}^c$ (also known as the exterior sphere condition).  
If there is $w \in \mathcal{R}(\tilde{\sigma}(x))$ with $w\cdot n(x)>0$, then $x$ is regular for $(x_t, V)$.     
\end{Corollary}
  
  \begin{proof}
  Without loss of generality, we can rotate and shift the set $V$ so that $x=0$ and $n(0)=e_d$.  In what follows, we thus assume $B_{\delta}(\delta e_d) \subset \overline{V}^c$ is tangent to $\partial V$ at $0$ and there 
   exists a vector $w \in \mathcal{R}(\tilde{\sigma}(0))$ such that $w\cdot n(0)=w_d>0$.  Since $w_d>0$, there exists $\lambda_0>0$  
  \begin{align*}
  \lambda w\in B_{\delta}(\delta e_d) \subset \overline{V}^c \qquad \textrm{for any} \quad \lambda \in (0, \lambda_0].  
  \end{align*}   
Hence, for the choice of $v= \lambda_0 w$, Proposition \ref{prop:invariance} part (qi)  is satisfied because $v \in \mathcal{R}(\tilde{\sigma}(0))$.  
Since $B_{\delta}(\delta e_d) $ is open, there exists $\delta'>0$ such that $B_{\delta'}(v) \subset  B_{\delta}(\delta e_d )$.   By convexity, whenever $v' \in  B_{\delta'}(v)$ then 
$\lambda v'\in B_{\delta} (\delta e_d)$ for all $\lambda \in (0,1]$.  Since $v' \in B_{\delta'}(v)$ if and only if $\sqrt{\epsilon \log \log \epsilon^{-1}} v' \in B_{\delta_{\epsilon}}(v_\epsilon)$ for any small 
$\epsilon >0$, property (qii) of Proposition~\ref{prop:invariance}  follows.  An application of  Proposition~\ref{prop:invariance} finishes the proof.  
   \end{proof}
         
         Using a nearly identical proof, we can also obtain the following result which is the so-called \emph{exterior cone condition}.  
         
  \begin{Corollary}
\label{cor:noisecone}
Let $x\in \partial V$ and suppose that $\tilde{b}\in C^\infty(\RR^d; \RR^d)$ and $\tilde{\sigma}\in C^\infty(\RR^d; M_{d\times k})$.  Suppose that $V$ satisfies the exterior cone condition at $x$; that is, there exists a basis $\{ x_1, x_2, \ldots, x_d \}$ of $\RR^d$ such that  
\begin{align*}
\text{\emph{Cone}}(x; \,x_1, \ldots, x_d): = \{ x+ \lambda_1 x_1 + \dots + \lambda_d x_d \, : \, \lambda_i \in (0,1) \} \subset \overline{V}^c.  
\end{align*}   
If the column space of $\tilde{\sigma}(x)$ contains a vector $w$ such that $x+w\in \text{\emph{Cone}}(x; \,x_1, \ldots, x_d)$, then $x$ is regular for $(x_t, V)$.     
\end{Corollary}   
 
            If $\mathcal{R}(\tilde{\sigma}(x))$ is not all of $\RR^d$, then identifying  regular points for $(x_t, V)$ on $\partial V$ can be complicated, because
 one has to know  the almost sure dynamics near the boundary point. Moreover, the method used in the proof of Proposition~\ref{prop:invariance} is not sufficient to characterize all points.  However, as the next examples 
 illustrates, the techniques developed here can be still useful.           
 
   \begin{Example}
   \label{ex:regIK}
 As in Example~\ref{ex:IK2}, we  consider the iterated Kolmogorov equation in dimension $d=2$ and assume $V=B_r(0)$ for a given $r>0$.  Corollary~\ref{cor:noise} implies that all points on $x_0=(x_1(0), x_2(0)) \in \partial V$ with $x_2(0) \neq 0$ have normal vector with non-zero second component, and therefore are regular for $(x_t, V)$.  On the other hand, if $x_2(0) = 0$, then
  $n(x_0)=(\pm 1, 0)$ and relations~\eqref{eqn:limsupIK} and~\eqref{eqn:liminfIK} with $x_2(0)=0$ imply that points $(\pm r, 0)$ are also regular for $(x_t, V)$.  Note that the same result holds if $V=B_r(0)^c$. 
   \end{Example}
   
   \begin{Example}
 Next, consider the same problem as in Example~\ref{ex:regIK} and Example~\ref{ex:IK2} but with $V\subset \RR^2$ assumed to be a general bounded open set with $C^\infty$ boundary $\partial V$.  
 Then, by Corollary~\ref{cor:noise}, all points $x_0 \in \partial V$, where the outward unit normal $n(x_0)=(n_1(x_0), n_2(x_0))$ to $\partial V$ has $n_2(x)\neq 0$ are regular.  
 If, on the other hand, $n(x_0)=(\pm 1,0)$ for some $x_0 \in \partial V$, then~\eqref{eqn:limsupIK} and~\eqref{eqn:liminfIK} 
  imply that $x_0$ is regular if and only if  $x_2(0)\geq 0$ and $n(x_0)=(+1, 0)$, or  $x_2(0)\leq 0$ and $n(x_0)=(-1,0)$.          
   \end{Example}

  \subsection{Modification of the boundary} If the domain $V$ is not apriori specified, the idea of this section is to slightly modify $V$ so that all points on the boundary are regular.  We do this using polygonal approximations under the assumption that there is noise in a uniform direction on the boundary and $V$ is convex and bounded with \emph{non-flat} $C^1$ boundary $\partial V$; that is, $\partial V$ is $C^1$ and for each $x \in \partial V$ and any $r > 0$, the set $\partial V \cap B(x, r)$ is not a subset of a hyperplane. The latter condition is satisfied, for example, if $V$ is strictly convex, or if at each 
  $x \in \partial V$ there is at least one non-zero principal curvature.

The construction of our polygonal approximations makes use of convex hulls of randomly chosen points on the boundary.  There are many different ways to do the selection of points, but here we do it according to Hausdorff measure    
 $\pi$ on $\partial V$.  That is, we will choose sufficiently many vertices independently and according to law of $\pi$.  Intuitively, for a large number of vertices, the convex hull of these points should be \emph{close} to $V$.  The main result in Sch\"{u}tt and Werner~\cite{SW_03} makes this precise on a set of high probability.  It turns out that such a resulting polygon cannot have, almost surely, any face parallel to the uniform direction in which the noise acts on the boundary.  Thus, we can then apply Corollary~\ref{cor:noisecone}.     
    
\begin{Remark}
Polygonal approximation of smooth domains is used, for example, in finite element method (FEM) to numerically solve differential equations.  
   \end{Remark}

   \begin{Theorem}
   \label{thm:convexhull}
 Let $d\geq 2$, $v\in \RR^d$ be a unit vector, and suppose that $V\subset \RR^d$ is convex, bounded and non-flat $C^1$ boundary $\partial V$.  Consider a probability space $(\widetilde{\Omega}, \widetilde{\mathcal{F}},\mathbf{Q})$ such that $\xi_1, \xi_2, \xi_3 \ldots$ are i.i.d. random variables defined on $(\widetilde{\Omega}, \widetilde{\mathcal{F}},\mathbf{Q})$ with distribution $\pi$.  Then, for every $\epsilon >0$, there exists $n(\epsilon) > 0$ and a set $S_\epsilon$ with $\mathbf{Q}(S_\epsilon) > 1 - \epsilon$ such that the following properties are satisfied:
 \begin{itemize}
 \item[(i)] For every collection of $d$-points in $\{ \xi_1, \xi_2, \ldots, \xi_{n(\epsilon)}\}$ there exists a unique hyperplane $H$ that contains those $d$ points.  Furthermore, $H$ is not parallel to $v$.  
 \item[(ii)]  If $[\xi_1, \xi_2, \ldots, \xi_{n(\epsilon)}]$ denotes the closed convex hull of $\xi_1, \xi_2, \ldots, \xi_{n(\epsilon)}$, then
 \begin{align*}
 |\overline{V}|- | [\xi_1, \xi_2, \ldots, \xi_{n(\epsilon)}]|< \epsilon \,,  
 \end{align*} 
where $|A |$ denotes the Lebesgue measure of $A$ in $\RR^d$.    
 \end{itemize}     
  
    \end{Theorem}

    In order to prove Theorem~\ref{thm:convexhull}, we first establish the following auxiliary result.
    
   \begin{Lemma}
   \label{lem:planeint}
   Suppose $d\geq 2$ and that $V\subset \RR^d$ is convex with $C^1$ non-flat boundary $\partial V$.  Then, any hyperplane in $\RR^d$ intersects $\partial V$ only on a set of $\pi$-measure zero.      
   \end{Lemma}
   
   \begin{Remark}
   Note that in this result, we may drop the hypothesis that $V$ is bounded.  
   \end{Remark}
   
   \begin{proof}[Proof of Lemma~\ref{lem:planeint}]
Without loss of generality, assume that the hyperplane $\mathcal{P}$ has normal vector $e_1=(1,0, \ldots, 0)$ and $\mathcal{P}\cap \partial V \neq \emptyset$.  Define
   \begin{align*}
   \ell_+&=\sup\{ \alpha \in \RR\,: \, \partial V \cap \{x\,: \, x_1 = \alpha \} \neq \emptyset \},\\
   \ell_- &= \inf\{  \alpha \in \RR\,: \, \partial V \cap \{x\,: \, x_1 = \alpha \} \neq \emptyset \}.   \end{align*} 
      If the set $\{ \alpha \in \RR\,: \, \partial V \cap \{x\,: \, x_1 = \alpha \}$ is not bounded above (respectively below), we set $\ell_+= \infty$ (respectively $\ell_-=-\infty)$.  Note that if $\ell_+\in \RR$, $\ell_+$ corresponds to the first coordinate of the \emph{rightmost} point on $\partial V$.  Similarly, if $\ell_{-} \in \RR$, $\ell_{-}$ corresponds to the first coordinate of the \emph{leftmost} point.  Observe that $\mathcal{P}\cap \partial V\neq \emptyset$ implies $\mathcal{P} = \{x : x_1 = \alpha\}$ for some $\alpha$ with $\alpha \in [\ell_-, \ell_+]$, $|\alpha| < \infty$. 
  
  First assume $\alpha = \ell_+< \infty$.  Without loss of generality, we may assume $\ell_+ = 0$ and $0 \in \mathcal{P}\cap \partial V$, otherwise we shift both sets.  To every $X \in \mathcal{P}\cap \partial V$, we associate its position vector $x = \overrightarrow{0X}$. Let $k$ be largest number of points $\{ X_1, X_2, \ldots, X_k \}\subset \mathcal{P}\cap \partial V$ such that the corresponding set of position vectors $\{x_1, x_2, \ldots, x_k \}$ is linearly independent.  If $k < d-1$, then  $\mathcal{L} := \textrm{span } \{x_1, \dots, x_k\} $ is $k \leq d-2$ dimensional linear space, and therefore 
$\partial V \cap \mathcal{P} \subset \mathcal{L} \cap \mathcal{P}$ is at most $(d-2)$-dimensional. Thus, $\pi(\partial V \cap \mathcal{P}) = 0$ as desired. 
If $k = d- 1$, then by the convexity of $V$ and the fact that $\alpha=\ell_+$, any convex combination of $\{ x_1, x_2, \ldots, x_{d-1}\}$ belongs to $\partial V \cap \mathcal{P}$.  However, this implies that $\partial V$ is locally a hyperplane about some point on $\partial V$, a contradiction to the fact that $\partial V$ is non-flat. 

A similar argument can be applied in the case when $\alpha = \ell_- >- \infty$. 
   
Finally, suppose $\alpha \in (\ell_{-}, \ell_{+})$. Fix $x\in \partial V \cap \mathcal{P}$ and note that the  normal vector to $\partial V$ at $x$ is not parallel to $e_1$.  Indeed, otherwise $\mathcal{P}$
is tangent to $V$ and by convexity, $V$ lies on one side of $\mathcal{P}$, a contradiction to $\alpha \in (\ell_{-}, \ell_{+})$.

As above, we can without loss of generality assume $\alpha = 0$ and $x = 0$. 
We parametrize $\partial V$ as $\partial V= \{ y\in \RR^d \, : \, \Phi(y)=0\}$ for some $C^1$-function $\Phi : \RR^d \to \RR$ with $\Phi(0)=0$ and, by the claim, 
$\nabla \Phi(0) \neq \lambda e_1$ for any $\lambda\in \RR$.  Then, 
\begin{equation}
 \partial V \cap \mathcal{P} \subset \{ y = (0, y_2, \dots, y_d): \Phi(0, y_2, \ldots, y_d)=0\} \,.
\end{equation}
To solve $\Phi(0, y_2, \ldots, y_d)=0$ we note that $\Phi(0) = 0$ and there exists $j\geq 2$ such that $\partial_{x_j} \Phi(0) \neq 0$. 
By the implicit function theorem, $y_j = \phi(y_2, \dots, y_{j - 1}, y_{j + 1}, \dots, y_{d - 1})$ locally for some $C^1$ function $\phi$, and therefore $\partial V \cap \mathcal{P}$ 
is locally a $(d -2)$-dimensional manifold. In particular $\pi(\partial V \cap \mathcal{P} \cap B_\rho) = 0$ for some $\rho > 0$. 

Since the countable union of set of zero measure is also set of zero measure, the result follows.  
\end{proof}

We are now ready to prove Theorem~\ref{thm:convexhull}.  
   
   \begin{proof}[Proof of Theorem~\ref{thm:convexhull}]
For any measurable $A_i \subset \partial V$ and any $n$, we have
\begin{align*}
\mathbf{Q}\{ \xi_{1}\in A_1, \ldots, \xi_n \in A_n \} = \pi(A_1) \pi(A_2)\dots \pi(A_n).  
\end{align*}
For any collection of points $y_1, \dots, y_k$ denote the set of vectors $\{\overrightarrow{y_1y_j}, j > 1\}$ by $\langle y_1, \dots, y_k\rangle $. 
Observe that $\langle y_1, \dots, y_k\rangle$ depends on the arrangement of points, but the span of the vectors $\langle y_1, \dots, y_k\rangle$ does not. Note that  $z$ belongs to the affine space 
defined by $y_1, \dots, y_k$ if an only if $z \in y_1 + \textrm{span} \{\langle y_1, \dots, y_k\rangle\} =: \mathcal{A}(y_1, \dots, y_k)$.
Also, we denote $y_1, \dots \hat{y}_j, \dots, y_k$
the sequence of points (or similarly vectors), where the point $y_j$ is omitted from the list. 
For $i_1, i_2, \dots, i_\ell\in \N$ all distinct with $\ell\leq d+1$, let
\begin{align}
D_{i_1 i_2 \dots i_\ell}= \{\omega \in \widetilde{\Omega} \, : \, \langle \xi_{i_1}(\omega),\xi_{i_2}(\omega), \dots, \xi_{i_\ell}(\omega)\rangle \text{ are linearly independent} \}
.\end{align}
We claim that $\mathbf{Q}(D_{i_1 i_2 \dots i_\ell})=1$. 
  Indeed,  since $i_1, i_2, \dots, i_\ell\in \N$ are all distinct with $\ell \leq d+1$, then by Lemma~\ref{lem:planeint} 
\begin{align*}
\mathbf{Q}(D_{i_1 i_2 \dots i_\ell}^c)&\leq \sum_{j=2}^\ell \mathbf{Q}\{\overrightarrow{\xi_{i_1}\xi}_{i_j} \in \text{span} \{\langle \xi_{i_1}, \dots, \hat{\xi}_{i_j}, \dots, \xi_{i_\ell}\rangle \} \}\\
&= \sum_{j=2}^\ell \int_{x_{i_j} \in \mathcal{A}(x_{i_1}, \dots,\hat{x}_{i_j}, \dots, x_{i_\ell} )\cap \partial V} \pi(dx_{i_j}) \pi^{\ell-1}(dx_{i_1} \dots d \hat{x}_{i_j} \dots dx_{i_\ell})=0 \,,
\end{align*}  
since $\mathcal{A}(x_{i_1}, \dots,\hat{x}_{i_j}, \dots, x_{i_\ell} )$ is at most 
$ (\ell-1)$-dimensional affine space which intersects $\partial V$ on a set of $\pi$-measure $0$ (see Lemma~\ref{lem:planeint}).

Fix a vector $v \in \RR^d$. 
If $\langle x_1, x_2, \dots, x_{d-1}\rangle$ are independent vectors, define
\begin{align*}
A_v(x_1, x_2, \dots, x_{d-1}) =
\{ x_{d} \in \partial V \, : \, v \in \textrm{span} \langle x_1, \dots, x_{d} \rangle \} 
\end{align*}
and if $\langle x_1, \dots, x_{d-1}\rangle$ are dependent define $A_v(x_1, x_2, \dots, x_{d-1})=\emptyset$.  
Note that if $v \not\in \langle x_1, x_2, \dots, x_{d}\rangle$ and $ \langle x_1, x_2, \dots, x_{d}\rangle$ are independent vectors, then $A_v(x_1, x_2, \dots, x_{d})$ is the intersection of $\partial V$ and the affine $(d - 1)$-dimensional space that contains points $ x_1, x_2, \dots, x_{d}$ and is parallel to $v$. 
For any distinct collection $i_1, i_2, \dots, i_d \in \N$, $D_{i_1 i_2 \dots i_{d}}$ is a a set of full measure, and therefore 
\begin{align*}
\mathbf{Q}(F_{i_1 i_2 \dots i_d})
&:= 
\mathbf{Q}(\{ \omega \in \widetilde{\Omega} \, : \, \mathcal{A}(\xi_{i_1}(\omega), \xi_{i_2}(\omega), \dots, \xi_{i_d}(\omega)) \text{ is parallel to } v \})\\
&= \mathbf{Q}(\{ \omega \in D_{i_1 i_2 \dots i_{d}} \, : \, \mathcal{A}(\xi_{i_1}(\omega), \xi_{i_2}(\omega), \dots, \xi_{i_d}(\omega)) \text{ is parallel to } v \}) \,.\end{align*}
Hence, 
by the Fubini-Tonelli theorem
\begin{align*}
\mathbf{Q}(F_{i_1 i_2 \dots i_d})
&=\mathbf{Q}(\{ \omega \in D_{i_1 i_2 \dots i_{d}}\, : \, \xi_{i_j}(\omega) \in A_v(\xi_{i_1}(\omega), \dots, \hat{\xi}_{i_j}(\omega), \dots, \xi_{i_{d}}(\omega)) \})\\
&=\int_{\{x_{j} \in A_{v}(x_1, \dots,\hat{x}_j, \dots,  x_{d})\}}\pi(dx_j) \pi^{d-1}(dx_1 \dots \hat{x}_j \dots dx_{d})=0 \,,
\end{align*}
where in the last equality we used  Lemma~\ref{lem:planeint}  and the fact that $A_{v}(x_1, \dots,\hat{x}_j, \dots,  x_{d})$ is at most $(d-1)$-dimensional affine 
space, and has $\pi$-measure zero when intersected with $\partial V$.

By taking finite unions, 
\begin{align}
F := \bigcup_{\substack{i_1, i_2, \dots, i_d \in\{1, \dots, N\} \\\text{all distinct}}} F_{i_1 i_2 \dots i_d},
\end{align}
we have $\mathbf{Q}(F)=0$.  Thus for any $N\geq d$ and any $\omega \in F^c$, with $\mathbf{Q}(F^c) = 1$ the realization
\begin{align*}
\xi_1(\omega), \xi_2(\omega), \dots, \xi_N(\omega)
\end{align*}
satisfies (i).  

To obtain (ii), fix $\epsilon >0$ and by~\cite[Theorem~1.1]{SW_03} there is a sufficiently large $N\geq d+1$ such that 
\begin{align*}
|V|- \E_\mathbf{Q} | [\xi_1, \xi_2, \dots, \xi_N] |= |V|-\E_\mathbf{Q} \mathbf{1}_{F^c} | [\xi_1, \xi_2, \dots, \xi_N] | < \epsilon^2.  
\end{align*}
Since convexity of $V$ implies$ | [\xi_1, \xi_2, \dots, \xi_N] | \leq |V|$, then by Chebyshev's inequality
 \begin{equation}
 \mathbf{Q}\{  |V|- | [\xi_1, \xi_2, \dots, \xi_N] | > \epsilon \} \leq \frac{|V|-\E_\mathbf{Q} \mathbf{1}_{F^c} | [\xi_1, \xi_2, \dots, \xi_N] |}{\epsilon} < \epsilon \,, 
 \end{equation}
which concludes the proof.

   \end{proof}

   An almost immediate consequence of the previous result is the following corollary.  
   
   \begin{Corollary}
   Suppose that $V\subset \RR^d$, $d \geq 2$ is a non-empty, bounded, convex, non-flat domain with $C^1$ boundary $\partial V$. 
    Assume $\tilde{b}\in C^\infty(\RR^d; \RR^d)$ and $\tilde{\sigma}\in C^\infty(\RR^d; M_{d\times k})$ and there exits a unit vector $v \in \RR^d$
    such that $v \in \mathcal{R}(\sigma(x))$ for every $x\in \partial V$.  
    Then, for every $\epsilon >0$, there exists  a non-empty, open convex domain $D_\epsilon \subset V$ with piecewise smooth boundary $\partial D_\epsilon$ such that every 
    $y\in \partial D_\epsilon$ is regular for $(x_t, D_\epsilon)$ and   $|V|-|D_\epsilon| < \epsilon$.    
      \end{Corollary}
      
      \begin{proof}
Fix $\epsilon >0$  and select points $x_1, x_2, \dots, x_{n(\epsilon)} \in \partial V$ satisfying both (i) and (ii) of Theorem~\ref{thm:convexhull}.  
Define $$D_\epsilon= \text{interior}([x_1,x_2, \dots, x_{n(\epsilon)}]).$$
      and note that $|V|-|D_\epsilon| < \epsilon$ and $\partial D_\epsilon$ is piecewise $C^2$.  Since no face on the boundary $\partial D_\epsilon$ is parallel to $v$ ,
      Corollary~\ref{cor:noisecone} implies the assertion.         
      \end{proof}

\bibliographystyle{amsplain}
\providecommand{\bysame}{\leavevmode\hbox to3em{\hrulefill}\thinspace}
\providecommand{\MR}{\relax\ifhmode\unskip\space\fi MR }
\providecommand{\MRhref}[2]{%
  \href{http://www.ams.org/mathscinet-getitem?mr=#1}{#2}
}
\providecommand{\href}[2]{#2}

\end{document}